\documentclass[12pt]{article}
\usepackage{color}

\usepackage{amscd} 
\usepackage{amsmath,amssymb}
\usepackage{amsthm}
\usepackage{graphicx}
\usepackage{enumerate}
\usepackage{epstopdf}
\usepackage{mathtools}
\usepackage{xr}
\numberwithin{equation}{section}
\numberwithin{figure}{section}
\newtheorem{theorem}{Theorem}[section]

\newtheorem{lemma}[theorem]{Lemma}
\newtheorem{proposition}[theorem]{Proposition}
\newtheorem{corollary}[theorem]{Corollary}
\newtheorem{fact}[theorem]{Fact}
\newtheorem{example}[theorem]{Example}
\theoremstyle{definition}

\theoremstyle{remark}
\newtheorem{remark}[theorem]{Remark}
\theoremstyle{conjecture}

\renewcommand{\setminus}{-}
\newcommand{\Abb}[4]{{\mathbb{A}}_{{#1},{#2},{#3}}^{{#4}}}
\newcommand{\Atbb}[4]
{{\widetilde{\mathbb{A}}}_{{#1},{#2},{#3}}^{{#4}}}
\newcommand{\Attbb}[4]
{\tilde{\tilde{{\mathbb{A}}}}_{{#1},{#2},{#3}}^{{#4}}}
\newcommand{\pr}[2]{\operatorname{pr}_{{#1}\to{#2}}}

\newcommand{\KC}{\widetilde{K}_{\lambda,\nu}^{\mathbb{C}}}

\newcommand{\Acal}[4]{{\mathcal{A}}_{{#1},{#2},{#3}}^{{#4}}}
\newcommand{\Atcal}[4]{\widetilde{{\mathcal{A}}}_{{#1},{#2},{#3}}^{{#4}}}

\newcommand{\Cbb}[3]{{\mathbb{C}}_{{#1},{#2}}^{{#3}}}
\newcommand{\Ctbb}[3]{\widetilde{{\mathbb{C}}}_{{#1},{#2}}^{{#3}}}
\newcommand{\Ccal}[3]{{\mathcal{C}}_{{#1},{#2}}^{{#3}}}
\newcommand{\Ctcal}[3]{\widetilde{{\mathcal{C}}}_{{#1},{#2}}^{{#3}}}
\newcommand{\Dcall}[2]{{\mathcal{D}}_{{#1}}^{{#2}}}
\newcommand{\Exterior}{\mathchoice{{\textstyle\bigwedge}}%
    {{\bigwedge}}%
    {{\textstyle\wedge}}%
    {{\scriptstyle\wedge}}}

\newcommand{\Ttbb}[3]{\widetilde{{\mathbb{T}}}_{{#1},{#2}}^{{#3}}}

\newcommand{\Ttcal}[3]{\widetilde{\mathcal{T}}_{{#1},{#2}}^{{#3}}}
\newcommand{\Ttcall}[2]{\widetilde{\mathcal{T}}_{{#1}}^{{#2}}}

\begin{document}
\title
{Residue formula
 for regular symmetry breaking operators
}
\author{Toshiyuki KOBAYASHI}
\date{}

\maketitle%

\abstract
{We prove an explicit residue formula
 for a meromorphic continuation of conformally covariant integral operators
 between differential forms
 on ${\mathbb{R}}^n$
 and on its hyperplane.  
The results provide a simple and new construction 
of the conformally covariant differential symmetry breaking operators
 between differential forms
 on the sphere 
 and those on its totally geodesic hypersurface
 that were introduced
 in [Kobayashi--Kubo--Pevzner, Lect.~Notes Math.~(2016)].  
Moreover, 
 we determine the zeros
 of the matrix-valued regular symmetry breaking operators
 between principal series representations
 of $O(n+1,1)$ and $O(n,1)$.}

\setcounter{tocdepth}{1}
\tableofcontents

\vskip 1pc
\section{Statement of the main results}
\label{subsec:residue}

Let ${\mathcal{E}}^i({\mathbb{R}}^{n})$ be the space
 of (complex-valued) differential $i$-forms on ${\mathbb{R}}^{n}$, 
 and ${\mathcal{E}}_c^i({\mathbb{R}}^{n})$ the subspace
 of those having compact support.  
The object of study in this article is a meromorphic family of operators
\[
\Abb \lambda \nu \pm {i,j}:
{\mathcal{E}}_c^i({\mathbb{R}}^{n}) \to {\mathcal{E}}^j({\mathbb{R}}^{n-1}), 
\]
which are initially defined as integral operators
 when ${\operatorname{Re}} \lambda \gg |{\operatorname{Re}} \nu|$.

The operators $\Abb \lambda \nu \pm {i,j}$ arise as
\begin{enumerate}
\item[$\bullet$] matrix-valued regular {\it{symmetry breaking operators}}
 for principal series representations 
 for strong Gelfand pair $(O(n+1,1),O(n,1))$;
\item[$\bullet$]
conformally covariant operators on differential forms
 on the model space $(S^n,S^{n-1})$;
\item[$\bullet$]
the formal adjoint of a deformation of matrix-valued Poisson transforms.  
\end{enumerate}

By choosing appropriate Gamma factors
 $a_{\pm}(\lambda,\nu)$ (see \eqref{eqn:aln}), 
 we renormalize 
 $\Abb \lambda \nu \pm {i,j}$
 by 
\[
  \Atbb \lambda \nu \pm {i,j}:= a_{\pm}(\lambda,\nu)\Abb \lambda \nu \pm {i,j}
\]
 so that $\Atbb \lambda \nu \pm {i,j}$ depend holomorphically
 on $(\lambda, \nu)$ in the entire plane ${\mathbb{C}}^2$, 
 see Fact \ref{fact:holo} below.

The goal of this paper is in twofold:
\begin{enumerate}
\item[$\bullet$] to find the residue formula of the matrix-valued operators $\Abb \lambda \nu \pm {i,j}$
 along $\nu-\lambda \in {\mathbb{N}}$
 (see Theorem \ref{thm:153316a});
\item[$\bullet$]
to determine all the (isolated) zeros of the normalized operators $\Atbb \lambda \nu \pm {i,j}$ (see Theorem \ref{thm:Avanish} and Remark \ref{rem:Avanish}).  
\end{enumerate} 

\subsection{Integral operators $\Abb \lambda \nu \pm {i,j} \colon {\mathcal{E}}^i({\mathbb{R}}^n) \to {\mathcal{E}}^j({\mathbb{R}}^{n-1})$}
\label{subsec:intmero}
To state our main results, 
 we fix some notation.  

Let 
$
  |x|:= (x_1^2 + \cdots + x_n^2)^{\frac 1 2}
$
 for $x \in {\mathbb{R}}^n$.  
We define 
\begin{equation}
\label{eqn:psix}
\psi \colon {\mathbb{R}}^{n}\setminus \{0\} \to O(n), 
\quad
x \mapsto I_n -\frac{2 x {}^{t\!} x}{|x|^2}, 
\end{equation}
 as a matrix expression of the reflection 
 with respect to normal vector $x$.  
Let $\sigma^{(i)} \colon O(n) \to GL_{\mathbb{C}}(\Exterior^i({\mathbb{C}}^n))$
 be the $i$th exterior representation 
 of the natural representation of $O(n)$ on ${\mathbb{C}}^n$.  

We identify ${\mathcal{E}}^i({\mathbb{R}}^{n})$
 with $C^{\infty}({\mathbb{R}}^{n}) \otimes \Exterior^i({\mathbb{C}}^n)$, 
 and similarly,
 ${\mathcal{E}}^j({\mathbb{R}}^{n-1})$
 with $
C^{\infty}({\mathbb{R}}^{n-1}) \otimes \Exterior^j({\mathbb{C}}^{n-1}).  
$
We define 
$
  \pr i j \colon \Exterior^{i}({\mathbb{C}}^{n}) \to 
                 \Exterior^{j}({\mathbb{C}}^{n-1})
$
 for $j=i-1,i$ to be the first and second projections
 of the decomposition 
$
   \Exterior^{i}({\mathbb{C}}^{n})
=
   \Exterior^{i-1}({\mathbb{C}}^{n-1})
\oplus
   \Exterior^{i}({\mathbb{C}}^{n-1})
$, 
 respectively, 
 so that the following linear map
\begin{equation}
\label{eqn:symb}
  {\operatorname{Rest}}_{x_n=0}
  \circ
  ({\operatorname{id}} \oplus \iota_{\frac{\partial}{\partial x_n}})
  \colon
  {\mathcal{E}}^i({\mathbb{R}}^n) \to {\mathcal{E}}^i({\mathbb{R}}^{n-1}) \oplus {\mathcal{E}}^{i-1}({\mathbb{R}}^{n-1})
\end{equation}
 is identified with 
$
{\operatorname{id}} \otimes (\pr ii \oplus \pr i {i-1}).  
$ 
Here $\iota_{\frac{\partial}{\partial x_n}}$ denotes the inner multiplication 
of the vector field 
 $\frac{\partial}{\partial x_n}$.

For $j \in \{i-1,i\}$ and ${\operatorname{Re}} \lambda \gg |{\operatorname{Re}}\nu|$, 
 we define $\operatorname{Hom}_{\mathbb{C}}(\Exterior^i({\mathbb{C}}^{n}), \Exterior^j({\mathbb{C}}^{n-1}))$-valued, 
 locally integrable functions $\Acal \lambda \nu \pm {i,j}$
 on ${\mathbb{R}}^n$ by
\begin{align}
\Acal \lambda \nu {+} {i,j}(x)
:=&
(x_1^2 + \cdots + x_n^2)^{-\nu}|x_n|^{\lambda+ \nu-n} \pr i j \circ \sigma^{(i)}(\psi(x)), 
\label{eqn:Aijp}
\\
\Acal \lambda \nu {-} {i,j}(x)
:=&
\Acal \lambda \nu {+} {i,j}(x) {\operatorname{sgn}}(x_n).  
\label{eqn:Aijm}
\end{align} 

We introduce a normalizing factor 
 $a_{\pm}(\lambda,\nu)$ by 
\begin{equation}
\label{eqn:aln}
  a_{\varepsilon (\kappa)}(\lambda,\nu)^{-1}
 :=
  \Gamma(\frac{\lambda+\nu-n+1+\kappa}{2})\Gamma(\frac{\lambda-\nu+\kappa}{2})
\end{equation} 
for $\kappa=0,1$, 
 where we set $\varepsilon \colon \{0,1\} \to \{\pm\}$
 by $\varepsilon (0) =+$ and $\varepsilon (1)=-$. 
Then $a_{\pm}(\lambda,\nu)$ are holomorphic functions
 of $(\lambda,\nu)$ in ${\mathbb{C}}^2$.  
We set 
\begin{equation}
\Atcal \lambda \nu {\pm} {i,j}
:=
a_{\pm}(\lambda,\nu) \Atcal \lambda \nu {\pm} {i,j}.  
\label{eqn:Aijn}
\end{equation}

\begin{fact}
[see \cite{sbon, sbonvec}.]
\label{fact:holo}
Let $j=i-1$ or $i$. 
Then the integral operators
\[
   \Atbb \lambda \nu \pm {i,j} \colon
   {\mathcal{E}}_c^i({\mathbb{R}}^n) \to {\mathcal{E}}^j({\mathbb{R}}^{n-1}),    f \mapsto {\operatorname{Rest}}_{x_n=0} \circ(\Atcal \lambda \nu \pm {i,j} \ast f), 
\]
originally defined when ${\operatorname{Re}} \lambda \gg |{\operatorname{Re}}\nu|$, 
 extend to continuous operators
 that depend holomorphically on $(\lambda,\nu)$
 in the entire complex plane ${\mathbb{C}}^2$.  
Moreover,
 $\{(\lambda,\nu) \in {\mathbb{C}}^2 : \Atbb \lambda \nu \pm {i,j}=0\}$
 is a discrete subset of ${\mathbb{C}}^2$.  
\end{fact}

\begin{remark}
The Gamma factors 
 $a_{\pm}(\lambda,\nu)$ in \eqref{eqn:Aijn} 
 are chosen in an optimal way
 that there is no pole of $\Atbb \lambda \nu \pm {i,j}$
 and that the zeros are of codimension two in ${\mathbb{C}}^2$.  
We note that $a_+(\lambda, \nu)$ coincides with normalizing factor
 of the scalar-valued regular symmetry breaking operator
 ${\mathbb{A}}_{\lambda,\nu}$ in \cite[(7.8)]{sbon}
 when $i=j=0$ and $\kappa=0$.  
\end{remark}

\subsection{Residue formula of matrix-valued operators
 $\Abb \lambda \nu \pm {i,j}$}

The first factor
 $\Gamma(\frac{\lambda + \nu -n +1+ \kappa}{2})$
 of $a_{\varepsilon(\kappa)}(\lambda,\nu)$ in \eqref{eqn:aln}
 arises  from the normalization 
 of the distribution
 $|x_n|^{\lambda+\nu-n}(\operatorname{sgn}x_n)^{\kappa}$
 of one-variable,
 and the corresponding residue of $\Abb \lambda \nu {\pm} {i,j}$
 is easily obtained.  
On the other hand, 
 the second factor
 $\Gamma(\frac{\lambda -\nu +\kappa}{2})$ is more involved
 because it arises not only from the normalization 
of $(x_1^2+ \cdots + x_n^2)^{-\nu}$ 
but from the whole $\Acal \lambda \nu \pm {i,j}$.  
The main result of this article
 is to give a closed formula for the residues 
 of the operators $\Abb \lambda \nu \pm{i,j}$ 
 at the places of the poles
 of $\Gamma(\frac{\lambda -\nu +\kappa}{2})$ as follows:  
\begin{theorem}
[residue formula of $\Abb \lambda \nu \pm{i,j}$]
\label{thm:153316a}
Let $j=i-1$ or $i$, 
 and $\kappa \in \{0,1\}$.  
Suppose $\nu-\lambda=2m+\kappa$
 with $m \in {\mathbb{N}}$.  
Then we have
\begin{equation}
\label{eqn:resAvecC}
\Atbb \lambda \nu \kappa {i,j}
=
\frac{(-1)^{i-j+m+\kappa}\pi^{\frac{n-1}2}m!}{2^{2m-1+3\kappa}\Gamma(\nu+1)}
\Cbb \lambda \nu {i,j}.  
\end{equation}
\end{theorem}

Here $\Cbb \lambda \nu {i,j} \colon {\mathcal{E}}^{i}({\mathbb{R}}^{n}) \to {\mathcal{E}}^{j}({\mathbb{R}}^{n-1})$
 is a matrix-valued differential operator
 introduced in \cite{KKP}.  
See \cite{KP1} for instance, 
 for the definition of differential operators
 between two manifolds.  
To review the differential operator $\Cbb \lambda \nu {i,j}$,
 we begin with a scalar-valued differential operator
 $\Ctbb \lambda \nu {}$ 
 which we call {\it{Juhl's operator}} from \cite{Juhl}.  
For $l:=\nu-\lambda \in {\mathbb{N}}$, 
 we set $m:=[\frac l 2]$, 
 the largest integer
 that does not exceed $\frac l 2$, 
 and define
 $\Ctbb \lambda \nu {}:C^{\infty}({\mathbb{R}}^{n}) \to C^{\infty}({\mathbb{R}}^{n-1})$ by 
\begin{equation}
\label{eqn:Cln}
  \Ctbb \lambda \nu {}
  :=
  \operatorname{Rest}_{x_n=0}
  \circ
  \sum_{k=0}^{m}
  \frac{\prod_{j=1}^{m-k+1}
  (\nu - \frac{n-1}{2} -m+j)}
  {2^{2k-l}k! (l-2k)!}
  (\Delta_{{\mathbb{R}}^{n-1}})^{k}
  (\frac{\partial}{\partial x_n})^{l-2k}.   
\end{equation}

More generally, 
 the matrix-valued operator
$
   \Cbb \lambda \nu {i,j}\colon {\mathcal{E}}^{i}({\mathbb{R}}^{n})
                           \to {\mathcal{E}}^{j}({\mathbb{R}}^{n-1})
$
 is defined by the following formula \cite[(2.24) and (2.26)]{KKP}:
\begin{align*}
\Cbb \lambda \nu {i, i}
:=&
\Ctbb{\lambda+1}{\nu-1}{} d d^{\ast}
-
\gamma(\lambda - \frac n 2, \nu -\lambda) 
\Ctbb{\lambda}{\nu-1}{} d \iota_{\frac{\partial}{\partial x_n}}
+\frac 1 2(\nu-i)
\Ctbb{\lambda}{\nu}{}, 
\\
\Cbb \lambda \nu {i,i-1}
:=&
- \Ctbb{\lambda+1}{\nu-1}{} d d^{\ast}
 \iota_{\frac{\partial}{\partial x_n}}
-
\gamma(\lambda - \frac {n-1} 2, \nu -\lambda) 
\Ctbb{\lambda+1}{\nu}{} d^{\ast}
+
\frac 1 2
(\lambda+i-n)
\Ctbb{\lambda}{\nu}{}\iota_{\frac{\partial}{\partial x_n}}, 
\end{align*}
where the codifferential 
$
   d^{\ast} \colon {\mathcal{E}}^{j+1}({\mathbb{R}}^n) \to {\mathcal{E}}^j({\mathbb{R}}^n)
$
 is the formal adjoint
 of the exterior derivative 
$
   d \colon {\mathcal{E}}^j({\mathbb{R}}^n) \to {\mathcal{E}}^{j+1}({\mathbb{R}}^n)
$.  
The constant $\gamma(\mu,a) \in {\mathbb{C}}$
 is given for $\mu \in {\mathbb{C}}$ and $a \in {\mathbb{N}}$ by 
\begin{equation}
\gamma(\mu,a)
:=
\begin{cases}
1
&\text{if $a$ is odd, }
\\
\mu + \frac a 2
\qquad
&\text{if $a$ is even.  }
\end{cases}
\label{eqn:gamma}
\end{equation}

\vskip 1pc
\par\noindent
{\bf{Notation.}}\enspace
For two subsets $A$ and $B$ of a set,
 we write
\[
  A \setminus B:= \{a \in A: a \not\in B\}
\]
rather than the usual notation 
$A \backslash B$.  

\[
{\mathbb{N}} =\{0,1,2,\cdots\},
\quad
{\mathbb{N}}_+ =\{1,2,\cdots\}.  
\]

\vskip 1pc
\par\noindent
{\bf{Acknowlegements.}}\enspace
This work was partially supported
 by Grant-in-Aid for Scientific Research (A) (25247006), 
 Japan Society for the Promotion
 of Science.

\section{Symmetry breaking in conformal geometry}
\label{sec:perspective}

We discuss the operators $\Atbb \lambda \nu \pm {i,j}$ from two different viewpoints:
 representation theory of real reductive groups
 (Section \ref{subsec:introSBO})
 and conformal geometry (Section \ref{subsec:conf}).  
With these perspectives,
 we shall explain two applications
 of Theorem \ref{thm:153316a} in Section \ref{subsec:residue}.  
Logically, 
 the results of this section are not used for the proof of Theorem \ref{thm:153316a}.

\subsection{Symmetry breaking for reductive groups}
\label{subsec:introSBO}
In this subsection, 
 we discuss the operators $\Atbb \lambda \nu {\pm}{i,j}$ from 
 the viewpoint of representation theory 
 of real reductive Lie groups.

Let $\Pi$ be a continuous representation of a group $G$
 on a topological vector space.  
If $G'$ is a subgroup of $G$, 
 we may think of $\Pi$ as a representation of the subgroup $G'$, 
 which is called the {\it{restriction}} of $\Pi$,  
 to be denoted by $\Pi|_{G'}$.  
Let $\pi$ be another representation 
 of the subgroup $G'$.  
A {\it{symmetry breaking operator}}
 is a continuous linear map $\Pi \to \pi$
 that intertwines the actions
 of the subgroup $G'$.

The operators $\Atbb \lambda \nu {\pm}{i,j}$ arise 
 as symmetry breaking operators
 for the pair $(G,G')=(O(n+1,1),O(n,1))$ as follows.  
Let $P=M A N$ be a minimal parabolic subgroup of $G$.  
For $0 \le i \le n$, 
 $\delta \in \{0,1\}$, 
 and $\lambda \in {\mathbb{C}}$, 
 we take $\Pi$ to be the
 (unnormalized) parabolic induction 
 $I_{\delta}(i,\lambda)={\operatorname{Ind}}_P^G(\Exterior^i({\mathbb{C}}^n) \otimes {\operatorname{sgn}}^{\delta} \otimes {\mathbb{C}}_{\lambda})$, 
where $\Exterior^i({\mathbb{C}}^n) \otimes {\operatorname{sgn}}^{\delta} \otimes {\mathbb{C}}_{\lambda}$
 stands for an irreducible representation of $P$
 that extends the outer tensor product representation
 of $M A \simeq O(n) \times O(1) \times {\mathbb{R}}$
 with trivial $N$-action.  
Likewise, 
 we take $\pi$ to be the parabolic induction 
$
   J_{\varepsilon}(j,\nu)={\operatorname{Ind}}_{P'}^{G'}
   (\Exterior^j({\mathbb{C}}^{n-1}) \otimes {\operatorname{sgn}}^{\varepsilon} \otimes {\mathbb{C}}_{\nu})
$
{} from a minimal parabolic subgroup $P'$ of $G'$.  
By identifying the open Bruhat cell of $G/P$
 with ${\mathbb{R}}^n$ and that of $G'/P'$ with ${\mathbb{R}}^{n-1}$, 
 we can realize $\Pi$ and $\pi$
 in ${\mathcal{E}}^i({\mathbb{R}}^n)$ and ${\mathcal{E}}^j({\mathbb{R}}^{n-1})$
 as the \lq\lq{$N$-picture}\rq\rq\ of principal series representations, 
respectively.  
Then the operators
$
  \Atbb \lambda \nu {\mu}{i,j} \colon
  {\mathcal{E}}_c^i({\mathbb{R}}^n) \to {\mathcal{E}}^j({\mathbb{R}}^{n-1})
$
 in Fact \ref{fact:holo}
 gives rise to symmetry breaking operators, 
 to be denoted by the same symbol, 
\[
    \Atbb \lambda \nu {\mu}{i,j} \colon I_{\delta}(i,\lambda) \to J_{\varepsilon}(j,\nu)
\]
for $\mu=+$ ($\delta\equiv \varepsilon \mod 2$)
 and $\mu=-$ ($\delta \not \equiv \varepsilon \mod 2$).  
If 
\begin{equation}
 \text{$\nu -\lambda \not \in 2 {\mathbb{N}}$
 for $\delta \equiv \varepsilon$}
\quad
 \text{or}
\quad
\text{$\nu -\lambda \not \in 2 {\mathbb{N}}+1$
 for $\delta \not \equiv \varepsilon$}, 
\label{eqn:generic}
\end{equation}
then the support of the distribution kernel 
 $\Atcal \lambda \nu {\mu}{i,j}$ has an interior point,
 and the operator $\Atbb \lambda \nu {\mu}{i,j}$ is called
 a {\it{regular}} symmetry breaking operator
 (\cite[Def.~3.3]{sbon}).

The dimension of the space 
 ${\operatorname{Hom}}_{G'}(I_{\delta}(i,\lambda)|_{G'}, 
J_{\varepsilon}(j,\nu))$
 of symmetry breaking operators
 is uniformly bounded with respect to the parameters
 $(\lambda,\nu,\delta,\varepsilon)$
 by the general theory 
 (\cite{xKOfm}).  
Moreover,
 it is one-dimensional and spanned by $\Atbb \lambda \nu \mu {i,j}$
 if the generic condition \eqref{eqn:generic} on parameters is satisfied, 
 see \cite{KobayashiSpeh}.

The complete classification
 of ${\operatorname{Hom}}_{G'}(I_{\delta}(i,\lambda)|_{G'}, 
J_{\varepsilon}(j,\nu))$
 for general $\lambda, \nu \in {\mathbb{C}}$
 and $\delta, \varepsilon \in \{0,1\}$ is accomplished in \cite{sbonvec}, 
 where a part of the proof for the exhaustion uses the results
 of this article, 
 namely,
 the vanishing criterion of $\Atbb \lambda \nu \pm {i,j}$
 in Theorem \ref{thm:Avanish}.  
We also refer to \cite{KobayashiSpeh}
 for the dimension formula
 of ${\operatorname{Hom}}_{G'}(I_{\delta}(i,\lambda)|_{G'}, 
J_{\varepsilon}(j,\nu))$, 
 and to \cite{KKP} for the classification of those operators
 that are given by differential operators
 such as $\Cbb \lambda \nu {i,i-1}$ and $\Cbb \lambda \nu {i,i}$.

\subsection{Conformally covariant symmetry breaking operators}
\label{subsec:conf}
We begin with the general setting
 where $(X,g)$ is a Riemannian manifold,
 and $Y$ is a submanifold endowed with the metric tensor $g|_Y$.  
We set
\begin{align*}
{\operatorname{Conf}}(X)
:=&
\{\text{conformal diffeomorphisms of $(X,g)$}\}, 
\\
{\operatorname{Conf}}(X;Y)
:=&
\{\varphi \in {\operatorname{Conf}}(X):
\varphi (Y)=Y\}.  
\end{align*}
Then there is a natural family
 of representations
 $\varpi_{u,\delta}^{(i)}$ of ${\operatorname{Conf}}(X)$
 with parameters $u \in {\mathbb{C}}$
 and $\delta \in {\mathbb{Z}}/2{\mathbb{Z}}$
 on the space ${\mathcal{E}}^i(X)$
 of differential $i$-forms
 for $0 \le i \le \dim X$
 (\cite[(1.1)]{KKP}, see also \cite{KO1}).

Likewise,
 there is a natural family of representations 
 $\varpi_{v,\varepsilon}^{(j)}$
 of the subgroup ${\operatorname{Conf}}(X;Y)$
 on the space ${\mathcal{E}}^j(Y)$
 of differential $j$-forms
 on the submanifold $Y$
 for $0 \le j \le \dim Y$.

The general question is to construct
 and classify continuous linear maps
 ({\it{conformally covariant symmetry breaking operators}})
 ${\mathcal{E}}^i(X) \to {\mathcal{E}}^j(Y)$
 that intertwine the restriction 
 $\varpi_{u,\delta}^{(i)}|_{{\operatorname{Conf}}(X;Y)}$
 and the representation $\varpi_{v,\varepsilon}^{(j)}$
 of the subgroup ${\operatorname{Conf}}(X;Y)$ of ${\operatorname{Conf}}(X)$.  
The integral operators $\Atbb \lambda \nu {\pm}{i,j}$
 and the differential operators $\Cbb \lambda \nu {i,j}$
 are such operators
 for the model space $(X,Y)=(S^n, S^{n-1})$.  

In fact,
 ${\operatorname{Conf}}(X)$ is locally isomorphic to $G=O(n+1,1)$
 and ${\operatorname{Conf}}(X;Y)$ is to $G'=O(n,1)$
 if $(X,Y)=(S^n, S^{n-1})$.  
Then the conformal representation 
 $(\varpi_{u,\delta}^{(i)}, {\mathcal{E}}^i(S^n))$
 of the group ${\operatorname{Conf}}(S^n)$ may be identified
 with the \lq\lq{$K$-picture}\rq\rq\ of the principal series 
representation $I_{\delta}(i,\lambda)$
 of $G=O(n+1,1)$
 after some shift of parameters
 (see \cite[Prop.~2.3]{KKP}).  
Similarly,
 the conformal representation
 $(\varpi_{v,\varepsilon}^{(j)}, {\mathcal{E}}^j(S^{n-1}))$
 of the subgroup ${\operatorname{Conf}}(S^n;S^{n-1})$
 is identified with $J_{\varepsilon}(j,\nu)$.  
Thus the integral operator $\Atbb \lambda \nu \pm {i,j}$ 
 and its holomorphic continuation give rise to conformally covariant,
 symmetry breaking operators
 ${\mathcal{E}}^i(S^n) \to {\mathcal{E}}^j(S^{n-1})$.

\subsection{Applications}
\label{subsec:appl}
Theorem \ref{thm:153316a} in Section \ref{subsec:residue}
 leads us to two applications:
\begin{enumerate}
\item[(1)]
(a necessary and sufficient condition for $\Atbb \lambda \nu {\pm} {i,j}$
 to vanish)
\enspace
The matrix-valued symmetry breaking operators
 $\Atbb \lambda \nu {\pm} {i,j}$ are defined as the holomorphic continuation
 of integral operators,
 and it is nontrivial to find the precise location
 of the zeros.  
By the residue formula
 (Theorem \ref{thm:153316a}), 
 we can determine the zeros of $\Atbb \lambda \nu {\pm} {i,j}$
 (see Theorem \ref{thm:Avanish} and Remark \ref{rem:Avanish}).  
This plays a crucial role
 in the classification problem
 of symmetry breaking operators
 (\cite{sbonvec}).  

\item[{\rm{(2)}}]
(another approach to construct conformally covariant differential operators)
\enspace
It is easy to see 
 that the integral transforms
 (and its analytic continuation)
 $\Atbb \lambda \nu \pm {i,j} \colon I_{\delta}(i,\lambda) \to J_{\varepsilon}(j,\nu)$
 respect the actions
 of the subgroup $G'=O(n,1)$ of $G=O(n+1,1)$
 (\cite{sbon, sbonvec}).  
Equivalently,
 $\Atbb \lambda \nu \pm {i,j}$ give conformally covariant,
 symmetry breaking operators from 
 ${\mathcal{E}}^i(S^n)$ to ${\mathcal{E}}^j(S^{n-1})$,
or from ${\mathcal{E}}^i({\mathbb{R}}^n)$ to ${\mathcal{E}}^j({\mathbb{R}}^{n-1})$, 
 as is seen in Section \ref{subsec:conf}.  
Thus the residue formula gives 
 a new proof
 that $\Cbb \lambda \nu {i,j}$ ($j=i,i-1$)
 is a conformally differential symmetry breaking operator from 
 ${\mathcal{E}}^i({\mathbb{R}}^n)$
 to ${\mathcal{E}}^j({\mathbb{R}}^{n-1})$, 
 for which the construction and classification were given in \cite{KKP}
 by using the F-method
 \cite{KOSS, KP1}.  
Indeed,
 the argument in this article does not use the F-method
 on which the main argument in \cite{KKP} relies.  
Since $\Cbb \lambda \nu {i,j}$ is recovered from its matrix coefficients
 by an elementary computation
 in differential geometry
 ({\it{cf}}.~Facts \ref{fact:153284} and \ref{fact:161478}), 
 the residue formula \eqref{eqn:resAvecC} in Theorem \ref{thm:153316a}
 reconstructs the conformally covariant,
 differential symmetry breaking operators $\Cbb \lambda \nu {i,j}$.  
\end{enumerate}

\begin{remark}
(scalar-valued case)
In the case where $i=j=0$, 
 the matrix-valued symmetry breaking operator 
 $\Cbb \lambda \nu {i,j}$ reduces to a scalar-valued one
 (Juhl's operator), 
 and we have
\[
  \Cbb \lambda \nu {0,0} = \frac 12 \nu \Ctbb \lambda \nu {}, 
\]
see \cite[p.~23]{KKP}.  
Thus Theorem \ref{thm:153316a} in this case
 coincides with \cite[Thm.~12.2 (2)]{sbon}, 
 see Fact \ref{fact:resAC}.  
Actually,
 our proof of Theorem \ref{thm:153316a} uses
 the results in the scalar case.  
\end{remark}

\section{Some identities in the Weyl algebra}
\label{subsec:Walg}

A key technique in our proof of the matrix-valued residue formula
 (Theorem \ref{thm:153316a}) relies on an algebraic manipulation
 in the Weyl algebra, 
 for which we give a basic set-up in this section.  
We shall develop it 
 for Juhl's operators in Section \ref{subsec:Weylalg}.

\vskip 1pc
Let ${\mathcal{D}}'({\mathbb{R}}^n)$ be the space 
  of distributions on $\mathbb{R}^n$,
 and ${\mathcal{D}}_{\{0\}}'(\mathbb{R}^n)$ the subspace
 consisting of distributions
 supported at the origin.  
Then the Weyl algebra 
\[
  {\mathbb{C}}[x,\frac{\partial}{\partial x}]
  \equiv
  {\mathbb{C}}[x_1, \cdots, x_n, \frac{\partial}{\partial x_1}, \cdots, \frac{\partial}{\partial x_n}]
\]
 acts naturally on ${\mathcal{D}}'({\mathbb{R}}^n)$
 and leaves the subspace ${\mathcal{D}}_{\{0\}}'(\mathbb{R}^n)$ invariant.  
Let ${\mathcal{J}}$ be 
the annihilator 
 of the Dirac delta function $\delta(x)=\delta(x_1, \cdots, x_n)$, 
 namely,
 the kernel of the following ${\mathbb{C}}[x,\frac{\partial}{\partial x}]$-homomorphism:
\begin{equation}
\label{eqn:Psi}
  \Psi:
  {\mathbb{C}}[x, \frac {\partial}{\partial x}]
  \to 
  {\mathcal{D}}'(\mathbb{R}^n), 
  \qquad
  P \mapsto P \delta.  
\end{equation}
Then ${\mathcal{J}}$ is the left ideal
 generated by the coordinate functions $x_1,$ $\cdots,$ $x_n$, 
 and $\Psi$ induces an isomorphism 
 of ${\mathbb{C}}[x, \frac {\partial}{\partial x}]$-modules.  
\begin{equation}
\label{eqn:WalgD}
\overline \Psi:
{\mathbb{C}}[x,\frac{\partial}{\partial x}]/ {\mathcal{J}}
\overset \sim \to 
{\mathcal{D}}_{\{0\}}'(\mathbb{R}^n).  
\end{equation}

Our strategy is to reduce (rather complicated) computations 
 in ${\mathcal{D}}_{\{0\}}'(\mathbb{R}^n)$ 
 to simpler algebraic ones 
 via the isomorphism \eqref{eqn:WalgD}
 by preparing systematically certain identities
 in the Weyl algebra ${\mathbb{C}}[x,\frac{\partial}{\partial x}]$ modulo 
${\mathcal{J}}$
 (see Lemmas \ref{lem:152327}, \ref{lem:170182} and \ref{lem:1532100}).

Before entering this part,
 we give the following observation:
\begin{lemma}
\label{lem:160205}
If $P$ is a differential operator on $\mathbb{R}^n$
 with constant coefficients, 
 then 
\begin{equation}
\label{eqn:160205}
  \Psi(P) \ast f= Pf
\qquad
\text{for all $f \in C^{\infty}(\mathbb{R}^n)$}.  
\end{equation}
\end{lemma}

\begin{proof}
Let $\alpha=(\alpha_1,\cdots, \alpha_n)$ be a multi-index, 
 and we write $P=\sum_{\alpha} a_{\alpha} \frac{\partial^{|\alpha|}}{\partial x^{\alpha}}$.  
Then 
\begin{align*}
\Psi (P) \ast f(x)
=& \int \sum_{\alpha} a_{\alpha} \frac{\partial^{|\alpha|} \delta}{\partial y^{\alpha}}(y) f(x-y) d y
\\
=& \sum_{\alpha} a_{\alpha} \int \delta(y) (-1)^{|\alpha|}
  \frac{\partial^{|\alpha|}}{\partial y^{\alpha}}f(x-y) d y
\\
=& \sum_{\alpha} a_{\alpha} 
   \frac{\partial^{|\alpha|}}{\partial x^{\alpha}}f(x).  
\end{align*}
Hence $\Psi (P) \ast f(x)= Pf$.  
\end{proof}

We set
\[
  \Delta \equiv \Delta_{\mathbb{R}^n}= \frac{\partial^2}{\partial x_1^2}
          + \cdots +
          \frac{\partial^2}{\partial x_n^2}.  
\]
\begin{lemma}
\label{lem:152327}
Let $k \in {\mathbb{N}}$, 
 and $1 \le p, q \le n$ with $p \ne q$.  
Then the following identities hold
 in the Weyl algebra ${\mathbb{C}}[x, \frac{\partial}{\partial x}]$
 modulo ${\mathcal{J}}$.  
\begin{enumerate}
\item[{\rm{(1)}}]
$x_p \Delta^k 
  \equiv
 -2 k \frac{\partial}{\partial x_p}
 \Delta^{k-1} \mod {\mathcal{J}}.  
$
\item[{\rm{(2)}}]
$x_p x_q \Delta^k 
 \equiv
 4k(k-1) \frac{\partial^2}{\partial x_p \partial x_q}
 \Delta^{k-2} \mod {\mathcal{J}}. 
$
\item[{\rm{(3)}}]
$x_p^2 \Delta^k 
 \equiv
 4k(k-1) \frac{\partial^2}{\partial x_p^2}
 \Delta^{k-2} 
 +
 2k \Delta^{k-1} \mod {\mathcal{J}}.  
$
\end{enumerate}
\end{lemma}
\begin{proof}
We denote by $[P,Q]:=P Q - Q P$
 the bracket of $P, Q \in {\mathbb{C}}[x, \frac{\partial}{\partial x}]$
 as usual.  
Then the assertions are derived from
 the following commutation relations:
\begin{align}
\label{eqn:152376}
   [x_p, \Delta^k]
   =&
   -2k \frac{\partial}{\partial x_p} \Delta^{k-1}, 
\\
\label{eqn:152386}
[x_p x_q, \Delta^k]
   =&
   4k(k-1) \frac{\partial^2}{\partial x_p \partial x_q} \Delta^{k-2}
   -2k
   (\frac{\partial}{\partial x_p} \Delta^{k-1} x_q
    +
    \frac{\partial}{\partial x_q} \Delta^{k-1} x_p), 
\\
\label{eqn:1523110}
[x_p^2, \Delta^k]
   =&
   4k(k-1) \frac{\partial^2}{\partial x_p^2} \Delta^{k-2}
   +2k \Delta^{k-1} 
   -4k
    \frac{\partial}{\partial x_p} \Delta^{k-1} x_p.  
\end{align}
The first equation \eqref{eqn:152376}
 is verified easily by induction on $k$.  
In turn, 
 the second and third ones \eqref{eqn:152386} and \eqref{eqn:1523110}
 follow from the iterated use
of \eqref{eqn:152376} 
 and from the identity 
 $[AB, C] = A[B, C] + [A, C]B$.  
\end{proof}

\section{Residue formul{\ae} of the matrix-valued Knapp--Stein intertwining operators}
\label{subsec:resBranson}
In this section we consider a baby-case
 ({\it{i.e.}} $G=G'$ case), 
 and apply the machinery in the previous section 
 to find a residue formula 
 for the matrix-valued Knapp--Stein intertwining operator
 $\Ttbb \lambda {n-\lambda}i$.  
Since the principal series representation $I_{\delta}(i,\lambda)$ is realized
 in the space ${\mathcal{E}}^i({\mathbb{R}}^n)$
 of differential forms on ${\mathbb{R}}^n$, 
 the residue formula should be given 
 by some familiar operators
 known in differential geometry.  
Actually, 
 we shall see 
 that the residue formula is proportional to Branson's conformal covariant differential operator
 \cite{B82}.

Our proof here illustrates an idea
 of the more complicated argument in later sections, 
 where we give a proof of our main results on symmetry breaking operators
 $\Atbb \lambda \nu {\pm}{i,j}$.

\subsection{Matrix-valued Knapp--Stein intertwining operators}
Suppose $0 \le i \le n$.  
For ${\operatorname{Re}}\lambda \gg 0$, 
 we define $\Ttcall \lambda i(x)$ to be
 an ${\operatorname{End}}_{\mathbb{C}}(\Exterior^i(\mathbb{C}^n))$-valued, 
 locally integrable function on ${\mathbb{R}}^n$
 by the following formula
\begin{equation}
\label{eqn:exrep}
   \Ttcall \lambda i (x)
:=
\frac{1}{\Gamma(\lambda-\frac n2)}|x|^{2(\lambda-n)}\sigma^{(i)}(\psi_n(x)).  
\end{equation}
The Knapp--Stein intertwining operator \cite{KS}
 between principal series representations
 of $G=O(n+1,1)$, 
\[
   \Ttbb \lambda {n-\lambda}i
   \colon
   I_{\delta}(i,\lambda) \to I_{\delta}(i,n-\lambda)
\qquad
\text{for $\delta \in \{0,1\}$, }
\]
is defined in the $N$-picture as the analytic continuation
 of the convolution map
\[
   {\mathcal{E}}_c^i({\mathbb{R}}^n) \to {\mathcal{E}}^i({\mathbb{R}}^n), 
\quad
   f \mapsto \Ttcall \lambda i \ast f.  
\]

Next, 
we recall that Branson's conformally covariant differential operator
(see \cite{B82})
\[
   \Dcall {2l}{(i)} \colon
   {\mathcal{E}}^i({\mathbb{R}}^n) \to {\mathcal{E}}^i({\mathbb{R}}^n)
\]
 is given by 
\begin{equation}
\label{eqn:Branson}
  \Dcall {2l}{(i)}
  :=
\begin{cases}
  -(\frac n 2-i) \Delta^l + l(d^{\ast}d -d d^{\ast})\Delta^{l-1}
\quad&\text{for $l \in {\mathbb{N}}_+$}, 
\\
  -(\frac n 2-i) {\operatorname{id}}
\quad
&\text{for $l =0$}, 
\end{cases}
\end{equation}
where $\Delta=-(d d^{\ast} + d^{\ast} d)$
 is the Laplace--Bertrami operator
 acting on differential forms.  
We adopt the normalization of $\Dcall {2l}{(i)}$
 given by \cite[(12.1)]{KKP}.  
In particular, 
 $\Dcall {2l}{(i)}$ vanishes 
 when $i=\frac n 2$ and $l=0$.  
The conformally covariant property of Branson's operator 
 is reformulated 
 as the intertwining property 
 between two principal series representations
 in their $N$-picture:
\[
    \Dcall {2l}{(i)} \colon I_{\delta}(i,\frac n 2-l) \to I_{\delta}(i,\frac n 2+l)
\quad\text{for $\delta \in \{0,1\}$.  }
\]
See \cite[Thm.~12.2]{KKP}
 for instance, for the classification
 of such operators.

Here is a relationship between 
 ${\operatorname{End}}_{\mathbb{C}}(\Exterior^i({\mathbb{C}}^n))$-valued
 Knapp--Stein intertwining operators
 $\Ttbb \lambda{n-\lambda} i$
 and Branson's conformally covariant operators $\Dcall {2l}{(i)}$:
\begin{theorem}
\label{thm:170183}
Let $0 \le i \le n$
 and $\lambda \in {\mathbb{C}}$.  
\begin{enumerate}
\item[{\rm{(1)}}]
The matrix-valued Knapp--Stein intertwining operator 
$\Ttbb \lambda{n-\lambda}i$ reduces to a differential operator
 if and only if $n-2\lambda \in 2{\mathbb{N}}$.  
\item[{\rm{(2)}}]
Suppose $l \in {\mathbb{N}}_+$.  
Then we have 
\begin{equation}
\label{eqn:KSres}
   \Ttbb \lambda{n-\lambda}i|_{\lambda=\frac n 2 -l}
   =
   \frac{(-1)^{l+1} \pi^{\frac n2}}{2^{2l} \Gamma(\frac n 2 + l + 1)} \Dcall {2l}{(i)}.  
\end{equation}
\end{enumerate}
\end{theorem}

Theorem \ref{thm:170183} will be proved
 in Section \ref{subsec:pfBran}
 after preparing some basic results.  

\subsection{Residue formula of the Riesz potential $|x|^{\mu}$}
We review a classical result on the Riesz potential 
 $|x|^{\mu}=(x_1^2+ \cdots +x_n^2)^{\frac \mu 2}$.  
This is a meromorphic family of distributions
 on ${\mathbb{R}}^n$, 
 and has simple poles 
 at $\mu=-n-2l$
 ($l \in {\mathbb{N}}$).  
Thus the normalized Riesz potential on ${\mathbb{R}}^n$
 defined by
\[
   \Ttcall \lambda{}(x)
   := \frac{1}{\Gamma(\lambda-\frac n 2)}|x|^{2(\lambda-n)}
\]
 depends holomorphically 
 on $\lambda$
 in the entire plane ${\mathbb{C}}$.  
The residue formula is classically known
(see \cite[Chap.\~(2.2)]{GelfandShilov}, {\cite{KS}} for example): 
\begin{fact}
[residue of $|x|^{\mu}$]
\label{fact:Rieszres}
Suppose $l \in {\mathbb{N}}$.  
Then we have:
\[
   \Ttcall{\frac n 2 - l}{}(x)=C(l,n) \Delta^l \delta(x),
\]
where we set 
\[
   C(l,n):=\frac{(-1)^l \pi^{\frac n 2}}{2^{2l} \Gamma(\frac n 2+l)}.  
\]
\end{fact}

\subsection{Index set ${\mathfrak{I}}_{n,i}$}
In what follows, 
we use the convention 
 of index sets as below.  
For $0 \le i \le n$, 
 we define
\begin{equation*}
{\mathfrak{I}}_{n,i}
:=\{I \subset \{1,\cdots,n\}
  : \# I =i\}.   
\end{equation*}
For $I = \{k_1, \cdots, k_i\}\in {\mathfrak{I}}_{n,i}$
 with $k_1 < \cdots < k_i$,
 we set 
\begin{align*}
    e_I:=& e_{k_1} \wedge e_{k_2} \wedge \cdots \wedge e_{k_i}, 
\\
  d x_I:=& dx_{k_1} \wedge dx_{k_2} \wedge \cdots \wedge dx_{k_i}.  
\end{align*}
Then $\{e_I\}$ forms a basis
 of the vector space $\Exterior^i({\mathbb{C}}^n)$, 
 and $\{d x_I\}$ forms a basis of ${\mathcal{E}}^i({\mathbb{R}}^n)$
 as a $C^{\infty}({\mathbb{R}}^n)$-module.  
We then have a natural isomorphism
 as $C^{\infty}({\mathbb{R}}^n)$-modules:
\begin{equation}
\label{eqn:EiRn} 
 C^{\infty}({\mathbb{R}}^n) \otimes \Exterior^i({\mathbb{C}}^n)
  \overset \sim \to 
 {\mathcal{E}}^i({\mathbb{R}}^n), 
\quad
  \sum f_I \otimes e_I
  \mapsto \sum f_I d x_I.   
\end{equation}

We introduce a family of quadratic polynomials $S_{I J}(x)$
 indexed by $I$, $J \in {\mathfrak{I}}_{n,i}$
 of $x=(x_1, \cdots,x_n)$ as follows:   
\begin{equation}
\label{eqn:SIJ}
   S_{I J}(x)
   =
   \begin{cases}
   \sum_{k=1}^n \varepsilon_I(k) x_k^2
   \qquad
   &\text{if $I=J$}, 
\\
   {\operatorname{sgn}}(I,I') x_p x_q
   \qquad
   &\text{if $\#(I \setminus J)=1$}, 
\\
   0
   \qquad
   &\text{if $\#(I \setminus J)\ge 2$}.  
   \end{cases}
\end{equation}

Here, 
 we set $\varepsilon_I(k)=1$
 for $k \in I$;
$=-1$ for $k \not \in I$.  
For $K \subset \{1,\cdots,n\}$ and $1 \le p,q \le n$, 
 ${\operatorname{sgn}}(K;p,q) \in \{\pm 1\}$ is defined by 
\[
   {\operatorname{sgn}}(K;p,q)
  :=
  (-1)^{\# \{r \in K: {\operatorname{min}} (p,q) < r < {\operatorname{max}} (p,q)\}}.  
\]

For $I,I' \in {\mathfrak{I}}_{n,i}$ with $\#(I \setminus I')=1$, 
we set 
\begin{equation}
\label{eqn:sgnII}
{\operatorname{sgn}}(I,I')
:={\operatorname{sgn}}(I \cap I'; p,q), 
\end{equation}
where $p \in I \setminus I'$
 and $q \in I' \setminus I$.

We have the following.  
\begin{lemma}
\label{lem:minordet}
Suppose $0 \le i \le n$.  
Then 
 the minor determinant of $\psi_n(x) \in O(n)$ 
 for $I,J \in {\mathfrak{I}}_{n,i}$ is given by
\begin{equation}
\label{eqn:psidet}
  (\det \psi_n(x))_{I J}= -\frac {1}{|x|^2} S_{I J}(x).  
\end{equation}
\end{lemma}

\subsection{Proof of Theorem \ref{thm:170183}}
\label{subsec:pfBran}
We complete the proof of Theorem \ref{thm:170183}
 by comparing the matrix components
 of the both sides of the equation \eqref{eqn:KSres}.  

For $I, J \in {\mathfrak {I}}_{n,i}$, 
 we define the $(I,J)$-component 
 of the matrix-valued distribution $\Ttcall \lambda i$ by
\[
   (\Ttcall \lambda i)_{I J} 
   :=
   \langle 
   \Ttcall \lambda i(e_I), e_J^{\vee}
   \rangle, 
\]
where $\{e_J^{\vee}\}$ is the dual basis.  
By \eqref{eqn:exrep} and \eqref{eqn:psidet}, 
 we have
\begin{align*}
   (\Ttcall \lambda i)_{I J}
   =&
   \frac{1}{\Gamma(\lambda-\frac n 2)}|x|^{2(\lambda-n)} (\det \psi_n (x))_{J I}
\\
=&
   \frac{-1}{\Gamma(\lambda-\frac n 2)}|x|^{2(\lambda-n-1)} S_{J I}(x)
\\
=&
   \frac{1}{\frac n 2-\lambda+1} \Ttcall {\lambda-1}{}(x) S_{I J}(x).  
\end{align*}
Applying Fact \ref{fact:Rieszres}, 
 we get the first statement
 of Theorem \ref{thm:170183}, 
 and the following equality
 in ${\mathcal{D}}_{\{0\}}'({\mathbb{R}}^n)$:
\begin{equation}
\label{eqn:Tres}
   (\Ttcall \lambda i)_{I J}|_{\lambda=\frac n 2-l}
   =
   \frac {1}{l+1}
   C(l+1,n)S_{I J}(x)\Delta^{l+1}\delta(x).  
\end{equation}
In order to compute the right-hand side
 of \eqref{eqn:Tres}, 
 we  work with the Weyl algebra by using Lemma \ref{lem:152327}.  
\begin{lemma}
\label{lem:170182}
Suppose $l \in {\mathbb{N}}_+$
 and $I,J \in {\mathfrak{I}}_{n,i}$.  
Then the following equalities hold
 in ${\mathbb{C}}[x, \frac{\partial}{\partial x}]/{\mathcal{J}}$.  
\begin{enumerate}
\item[{\rm{(1)}}]
Suppose $I=J$.  
\[
  S_{I J}(x)\Delta^{l+1}
  \equiv
  4(l+1)l(\sum_{p \in I} \varepsilon_I(p)\frac{\partial^2}{\partial x_p^2})
  \Delta^{l-1}
  + 2 (l+1)(2i-n)\Delta^{l}.  
\]
\item[{\rm{(2)}}]
Suppose $\#(I-J)=1$.  
Let $\{p,q\}:=(I \cup J) \setminus (I \cap J)$.  
\[
  S_{I J}(x)\Delta^{l+1}
  \equiv
  8 l (l+1) \operatorname{sgn}(I,J) \frac{\partial^2}{\partial x_p \partial x_q}  \Delta^{l-1}.  
\]
\end{enumerate}
\end{lemma}
\begin{proof}
Direct from the definition \eqref{eqn:SIJ}
 of $S_{I J}(x)$ and from Lemma \ref{lem:152327}.  
\end{proof}

By Lemma \ref{lem:170182}, 
 the equation \eqref{eqn:Tres} gives the following.  
\begin{lemma}
\label{lem:170113}
Suppose $l \in {\mathbb{N}}_+$.  
\begin{align*}
&(\Ttcall{\lambda}{(i)})_{I J}|_{\lambda=\frac n 2-l}
\\
   =&
   4 C(l+1,n) \times
   \begin{cases}
   l (\sum_{p \in I} \varepsilon_I(p) \frac{\partial^2}{\partial x_p^2})
   \Delta^{l-1}
   +
   (i-\frac n 2)\Delta^{l}
   \quad
   &\text{if $I=J$}, 
\\
   2 l {\operatorname{sgn}}(I,J) 
   \quad
   &\text{if $\#(I \setminus J)=1$}, 
\\
   0
   \quad
   &\text{otherwise}.  
   \end{cases}
\end{align*}
\end{lemma}

On the other hand, 
 Branson's conformally covariant differential operators
 $\Dcall {2l}{(i)}$ take the following form:
\begin{lemma}
\label{lem:170181}
Suppose $0 \le i \le n$ and $l \in {\mathbb{N}}_+$. 
For any $\alpha= f d x_I \in {\mathcal{E}}^i({\mathbb{R}}^n)$
 with $f \in C^{\infty}({\mathbb{R}}^n)$
 and $I \in {\mathfrak{I}}_{n,i}$, 
 we have  
\begin{align*}
  \Dcall {2l}{(i)} \alpha
=&
  (  (i-\frac n 2) \Delta^l f 
   + l(\sum_{p \in I} \varepsilon_I(p)\frac{\partial^2}{\partial x_p^2})
     \Delta^{l-1} f) d x_I
\\
  &+ 2 l \sum_{\substack{p \in I \\ q \not \in I}} \operatorname{sgn}(I;p,q)
  \frac{\partial^2}{\partial x_p \partial x_q}
  (\Delta^{l-1}f)
   d x_{I \setminus \{p\} \cup \{q\}}.  
\end{align*}
\end{lemma}

\begin{proof}
The formula is direct from the definition \eqref{eqn:Branson}
 of $\Dcall {2l}{(i)}$ 
 and from the following elementary computations:
\begin{align*}
dd^{\ast}(f d x_I)
=& - \sum_{p \in I} \frac{\partial^2 f}{\partial x_p^2} d x_I
 - \sum_{\substack {p \in I \\ q \not \in I}} 
   \operatorname{sgn}(I;p,q)
   \frac{\partial^2 f}{\partial x_p \partial x_q}
   d x_{I \setminus \{p\} \cup \{q\}}, 
\\
d^{\ast} d(f d x_I)
=& - \sum_{q \not \in I} \frac{\partial^2 f}{\partial x_q^2} d x_I
 + \sum_{\substack {p \in I \\ q \not \in I}} \operatorname{sgn}(I;p,q)
   \frac{\partial^2 f}{\partial x_p \partial x_q}
   d x_{I \setminus \{p\} \cup \{q\}}.  
\end{align*}
\end{proof}

\begin{proof}
[Proof of Theorem \ref{thm:170183}]
We identify ${\mathcal{E}}^i({\mathbb{R}}^n)$
 with $C^{\infty}({\mathbb{R}}^n) \otimes \Exterior^i({\mathbb{C}}^n)$
 via \eqref{eqn:EiRn}.  
Comparing the $(I, J)$-component
 of Branson's conformally covariant operators
 with that of the matrix-valued Knapp--Stein intertwining operator
 $\Ttbb \lambda {n-\lambda}i$
 in Lemmas \ref{lem:170113} and \ref{lem:170181}, 
 we get Theorem \ref{thm:170183}.  
\end{proof}

\subsection{Vanishing condition of $\Ttbb \lambda {n-\lambda} i$}
As a biproduct of the residue formula \eqref{eqn:KSres}, 
 we obtain the vanishing condition
 of the matrix-valued Knapp--Stein intertwining operator
 $\Ttbb \lambda{n-\lambda}i$ as below.  
\begin{corollary}
\label{cor:Tvanish}
$\Ttbb \lambda{n-\lambda}i=0$
 if and only if $n$ is even
 and $i=\lambda=\frac n 2$.  
\end{corollary}
\begin{proof}
Since both $\Ttcall \lambda{}(x)$ and 
$\sigma^{(i)} (\psi_n(x))$
 are smooth in ${\mathbb{R}}^n\setminus \{0\}$, 
 the distribution kernel $\Ttcall \lambda i$ vanishes
 only when ${\operatorname{Supp}} \Ttcall \lambda{} \subset \{0\}$, 
 or equivalently,
 $n-2\lambda \in 2 {\mathbb{N}}$.  
Suppose now $n-2\lambda=2l$
 for some $l \in {\mathbb{N}}$.  
Then it follows from Theorem \ref{thm:170183}
 that $\Ttcal \lambda{n-\lambda}{(i)}$ vanishes
 if and only if ${\mathcal{D}}_{2l}^{(i)}=0$.  
In turn, 
 this happens exactly 
when $i=\frac n 2$
 and $l=0$
 by \eqref{eqn:Branson}.  
Thus the corollary is proved.  
\end{proof}

\section{Juhl's operator in the Weyl algebra}
\label{subsec:Weylalg}
In order to prove our main results
 (Theorem \ref{thm:153316a}), 
 we need some further identities in ${\mathbb{C}}[x, \frac{\partial}{\partial x}]/{\mathcal{J}}$
 where ${\mathcal{J}}$ is the left ideal 
generated by $x_1$, $\cdots$, $x_n$
 as in Section \ref{subsec:Walg}.

For $l \in {\mathbb{N}}$, 
 we define a finite-dimensional vector space
 of polynomials of one variable
 by 
\[
  {\operatorname{Pol}}_l[z]_{\operatorname{even}}
   :=
  {\mathbb{C}}\text{-span}
  \langle z^{l-2j}: 0 \le 2j \le l \rangle.  
\]
We inflate a polynomial $g(z) \in {\operatorname{Pol}}_l[z]_{\operatorname{even}}$
 to a polynomial of two variables $s$ and $t$
 by 
\begin{equation}
\label{eqn:Ilg}
(I_l g)(s,t):=s^{\frac l 2} g \left(\frac{t}{\sqrt s}\right).  
\end{equation}
If $g(z)$ is given as $g(z)=\sum_{j=0}^{[\frac l2]} a_j z^{l-2j}$, 
 then the definition \eqref{eqn:Ilg} yields
\begin{equation}
\label{eqn:Ilgst}
 (I_l g)(s,t)=s^{\frac l 2}g \left(\frac{t}{\sqrt s}\right)=\sum_{j=0}^{[\frac l 2]} a_j s^j t^{l-2j}.  
\end{equation}
We note that 
$(I_l g)(s^2, t)$ is a homogeneous polynomial
 of $s$ and $t$ of degree $l$.  
The following example reveals 
 how Juhl's conformally covariant differential operators
 \cite{Juhl, KOSS} arise.

\begin{example}
\label{ex:PsiC}
Let $\widetilde C_l^{\alpha}(z)$ be the renormalized Gegenbauer polynomial
 defined by 
\[
   \widetilde C_l^{\alpha}(z)
   =
  \frac{1}{\Gamma(\alpha+ [\frac{l+1}{2}])}
  \sum_{k=0}^{[\frac{l}{2}]}
  (-1)^k
  \frac{\Gamma(l-k+\alpha)}{k! (l-2k)!}
  (2z)^{l-2k}.  
\]
We adopt the normalization given in \cite[(14.3)]{KKP}
 so that $\widetilde C_l^{\alpha}(z)$ is 
 a nonzero element in ${\operatorname{Pol}}_l[z]_{\operatorname{even}}$
 for any $\alpha \in {\mathbb{C}}$.  
Suppose $\nu - \lambda \in {\mathbb{N}}$.  
Then Juhl's conformally covariant operator $\Ctbb \lambda \nu{}$
 (see \eqref{eqn:Cln})
 is expressed as 
\[
   \Ctbb \lambda \nu{}= {\operatorname{Rest}}_{x_n=0} \circ P, 
\]
where $P$ is a homogeneous differential operator
 of order $\nu-\lambda$ on ${\mathbb{R}}^n$
 given by 
\[
   P:=(I_{\nu-\lambda}\widetilde C_{\nu-\lambda}^{\lambda - \frac{n-1}{2}})(-\Delta_{\mathbb{R}^{n-1}}, \frac{\partial}{\partial x_n}).  
\]

We also define a distribution on ${\mathbb{R}}^n$
 supported at the origin by
\begin{equation}
\label{eqn:PsiJuhl}
   \Ctcal{\lambda}{\nu}{}
   :=
   (I_{\nu-\lambda} \widetilde C_{\nu-\lambda}^{\lambda -\frac{n-1}{2}})
   (-\Delta_{{\mathbb{R}}^{n-1}},\frac{\partial}{\partial x_n}) 
     \delta(x_1, \cdots, x_n).  
\end{equation}
This is the distribution kernel 
 of the scalar-valued, 
 differential symmetry breaking operator $\Ctbb \lambda \nu {}$
 (see \eqref{eqn:Cln}), 
 and was denoted by $\KC$
in \cite[(10.3)]{sbon}
when $\nu- \lambda \in 2{\mathbb{N}}$.  
In terms of the map 
 $\Psi \colon {\mathbb{C}}[x,\frac{\partial}{\partial x}] \to {\mathcal{D}}_{\{0\}}({\mathbb{R}}^n)$ defined in \eqref{eqn:Psi}, 
 Lemma \ref{lem:160205} shows
\begin{align*}
  \Ctcal \lambda \nu {}=&\Psi (P), 
\\
  \Ctbb \lambda \nu{}=&\operatorname{Rest}_{x_n=0} \circ \,\Ctcal \lambda \nu{} \ast.  
\end{align*}
\end{example}
We begin with some basic computations
 in the Weyl algebra ${\mathbb{C}}[x,\frac{\partial}{\partial x}]$
 modulo the left ideal ${\mathcal{J}}$.  
\begin{lemma}
\label{lem:1532100}
Suppose $g(z)\in \operatorname{Pol}_l[z]_{\operatorname{even}}$.  
Let $\theta= z \frac{d}{d z}$.  
\begin{enumerate}
\item[{\rm{(1)}}]
For any $p$ with $1 \le p \le n-1$, 
 we have
\[
x_p (I_l g)(-\Delta_{\mathbb{R}^{n-1}},\frac{\partial}{\partial x_n})
\equiv
\frac{\partial}{\partial x_p} I_{l-2}((l-\theta) g)(-\Delta_{\mathbb{R}^{n-1}},\frac{\partial}{\partial x_n})
\mod {\mathcal{J}}.  
\]
\item[{\rm{(2)}}]
$
x_n (I_l g)(-\Delta_{\mathbb{R}^{n-1}},\frac{\partial}{\partial x_n})
\equiv
-(I_{l-1} g')(-\Delta_{\mathbb{R}^{n-1}},\frac{\partial}{\partial x_n})
\mod {\mathcal{J}}.  
$
\end{enumerate}
\end{lemma}
\begin{proof}
Applying $\frac{\partial}{\partial s}$ and $\frac{\partial}{\partial t}$
 to the equation \eqref{eqn:Ilgst}, 
 we get
\begin{alignat}{2}
\label{eqn:sIg}
\frac {\partial}{\partial s} (I_l g)(s,t)
&= \frac 1 2 I_{l-2} ((l-\theta) g)(s,t)
&&= \sum_j a_j j s^{j-1} t^{l-2j}, 
\\
\label{eqn:tIg}
\frac {\partial}{\partial s} (I_l g)(s,t)
&=
(I_{l-1} g')(s,t)
&&=\sum_{j} a_j(l-2j) s^j t^{l-2j-1}.  
\end{alignat}
\par\noindent
(1)\enspace
By Lemma \ref{lem:152327}, 
\begin{equation}
\label{eqn:153432}
x_p (I_l g)(-\Delta_{\mathbb{R}^{n-1}},\frac{\partial}{\partial x_n})
\equiv
2 \frac{\partial}{\partial x_p}\sum_j a_j j (-\Delta_{\mathbb{R}^{n-1}})^{j-1}
(\frac{\partial}{\partial x_n})^{l-2j}
\mod {\mathcal{J}}.  
\end{equation}
By \eqref{eqn:sIg}, 
 the right-hand side of \eqref{eqn:153432} amounts to
\[
\frac {\partial}{\partial x_p}I_{l-2} ((l-\theta) g)(-\Delta_{\mathbb{R}^{n-1}},\frac{\partial}{\partial x_n})\mod {\mathcal{J}}.  
\]
\par\noindent
(2)\enspace
Since
$x_n (\frac{\partial}{\partial x_n})^{l-2j} 
\equiv -(l-2j) (\frac{\partial}{\partial x_n})^{l-2j-1} \mod {\mathcal{J}}$, 
 we have
\begin{equation}
\label{eqn:1534}
x_n (I_l g)(-\Delta_{\mathbb{R}^{n-1}},\frac{\partial}{\partial x_n})
\equiv
-\sum_j a_j (l-2j) (-\Delta_{\mathbb{R}^{n-1}})^j
(\frac{\partial}{\partial x_n})^{l-2j}
\mod {\mathcal{J}}.  
\end{equation}
By \eqref{eqn:tIg}, 
 we get the desired formula.  
\end{proof}

\begin{lemma}
\label{lem:1532101}
Suppose $\lambda, \nu \in {\mathbb{C}}$
 with $\nu-\lambda \in {\mathbb{N}}$.  
Let $\Ctcal \lambda \nu {}$ be the distribution 
 on ${\mathbb{R}}^{n-1}$ 
 defined in \eqref{eqn:PsiJuhl}.  
Then the following equations hold.  
\begin{align}
\label{eqn:160170}
x_p \Ctcal \lambda \nu {} 
=& - 2 \frac{\partial}{\partial x_p} \Ctcal {\lambda+1}{\nu-1}{}
\qquad
\text{for $1 \le p \le n-1$, }
\\
\label{eqn:1532101}
x_n \Ctcal \lambda \nu {} =& - 2 \gamma(\lambda-\frac{n-1}{2}, \nu-\lambda)
                             \Ctcal {\lambda+1}{\nu}{}.  
\end{align}
\end{lemma}

\begin{proof}
We recall from \cite{AAR99}, or \cite[(14.8)]{KKP} and 
\cite[(A.13), (A.14)]{KOSS}
 that the normalized Gegenbauer polynomial 
 $\widetilde C_l^{\alpha}(z)$ satisfies the following relations
\begin{align}
\label{eqn:diffzC}
  (l-z \frac{d}{d z}) \widetilde C_l^{\alpha}(z) 
  =& -2 \widetilde C_{l-2}^{\alpha+1}(z), 
\\
\label{eqn:diffC}
\frac d {dz} \widetilde C_l^{\alpha}(z)
=&
2 \gamma (\alpha,l)
\widetilde C_{l-1}^{\alpha+1}(z), 
\end{align}
for $\alpha \in {\mathbb{C}}$
 and $l \in {\mathbb{N}}_+$.

Applying Lemma \ref{lem:1532100} (1) 
 to $g(z)= \widetilde C_l^{\alpha}(z)$, 
 we get from \eqref{eqn:diffzC} the following identity 
in ${\mathbb{C}}[x, \frac{\partial}{\partial x}]/{\mathcal{J}}$:
\[
   x_p (I_l \widetilde C_l^{\alpha})(-\Delta_{\mathbb{R}^{n-1}}, \frac{\partial}{\partial x_n})
\equiv -2 \frac{\partial}{\partial x_p}
       (I_{l-2} \widetilde C_{l-2}^{\alpha+1})(-\Delta_{\mathbb{R}^{n-1}}, \frac{\partial}{\partial x_n})
\mod {\mathcal{J}}.  
\]
Hence
 \eqref{eqn:160170} is verified by putting $\alpha = \lambda - \frac{n-1}{2}$
 and $l = \nu-\lambda$.  
Likewise,
 applying Lemma \ref{lem:1532100} (2)
 to $g(z) = \widetilde C_{\nu-\lambda}^{\lambda-\frac{n-1}{2}}(z)$, 
 we get \eqref{eqn:1532101} from \eqref{eqn:diffC}.  
\end{proof}

\begin{proposition}
\label{prop:152383}
Suppose $1 \le p, q \le n-1$, 
 and $\lambda, \nu \in {\mathbb{C}}$.  
Then we have the following equations
 in ${\mathcal{D}}'({\mathbb{R}}^n)$.  
\begin{align}
\label{eqn:152383n}
   x_p x_n \Ctcal{\lambda-1}{\nu+1}{}
   =&
   4 \gamma(\lambda-\frac{n-1}{2},\nu-\lambda) \frac{\partial}{\partial x_p}
   \Ctcal{\lambda+1}{\nu}{}.  
\\
x_p x_q \Ctcal{\lambda-1}{\nu+1}{}
=&
4 \frac{\partial^2}{\partial x_p \partial x_q}\Ctcal{\lambda+1}{\nu-1}{}.  
\label{eqn:1523101}
\\
x_p^2 \Ctcal{\lambda-1}{\nu+1}{}
=&
4 \frac{\partial^2}{\partial x_p^2}\Ctcal{\lambda+1}{\nu-1}{}
+
2 \Ctcal{\lambda}{\nu}{}.  
\label{eqn:1523114}
\\
x_n^2 \Ctcal{\lambda-1}{\nu+1}{}
=&
4 (\lambda - \frac {n-1}{2} + [\frac{\nu-\lambda+1}{2}])
  \Ctcal{\lambda+1}{\nu+1}{}  
\notag
\\
=& 2 (2 \nu - n +1)\Ctcal{\lambda}{\nu}{}
   -4 \Delta_{\mathbb{R}^{n-1}} \Ctcal{\lambda+1}{\nu-1}{}.  
\label{eqn:152403}
\\
Q_{n-1}(x) \Ctcal{\lambda-1}{\nu+1}{}
=&
4 \nu \Ctcal{\lambda}{\nu}{}
-4 (\lambda-\frac{n-1}{2}+[\frac{\nu-\lambda+1}{2}])
\Ctcal{\lambda+1}{\nu+1}{}.  
\label{eqn:160201}
\end{align}
\end{proposition}

\begin{proof}
The iterated use of Lemma \ref{lem:1532101}
 yields these identities except for the second equality
 in \eqref{eqn:152403}.  
The second equality is nothing but the three-term relation \cite[(9.10)]{KKP}
\[
(\lambda-\frac{n-1}{2} + [\frac{\nu-\lambda+1}{2}])\Ctcal{\lambda+1}{\nu+1}{}
=
(\nu - \frac{n-1}{2}) \Ctcal{\lambda}{\nu}{}
- \Delta_{\mathbb{R}^{n-1}} \Ctcal{\lambda+1}{\nu-1}{}, 
\]
which is derived from the following three-term relation 
 of the Gegenbauer polynomials:
\begin{equation}
\label{eqn:152563}
(\mu+a) \widetilde C_a ^{\mu}(t)
+\widetilde C_{a-1} ^{\mu+1}(t)
=(\mu+[\frac{a+1}{2}])
\widetilde C_a ^{\mu+1}(t).  
\end{equation}
\end{proof}

\section{Reduction to the scalar-valued case}
\label{sec:PropC}
Our strategy to find the residue of the matrix-valued operators
 $\Abb \lambda \nu {\varepsilon}{i,j}$ ($\varepsilon = \pm$) is
 to reduce it to the scalar-valued case 
 by giving an expression 
 of the following form
 (see Lemma \ref{lem:AIJ} below):  
\begin{equation}
\label{eqn:QHA}
  \Abb \lambda \nu {\varepsilon}{i,j}
  =
  Q(\lambda,\nu)^{-1} H(x) \Atbb {\lambda+a}{\nu-b}+{}.  
\end{equation}
where the primary features of the three factors are as follows:
\begin{enumerate}
\item[$\bullet$]
$Q(\lambda,\nu)$ is a polynomial of $(\lambda,\nu)$;
\item[$\bullet$]
$H(x)$ is a ${\operatorname{Hom}}_{\mathbb{C}}(\Exterior^i({\mathbb{C}}^n), \Exterior^j({\mathbb{C}}^{n-1}))$-valued
 polynomial of $x$;
\item[$\bullet$]
$\Atbb {\lambda+a}{\nu-b}+{}$ is the {\it{scalar-valued}} regular symmetry breaking operator
 by shift $(a,b) \in {\mathbb{N}}^2$.  
\end{enumerate}

\subsection{Polynomial $g_{I J}(x)$}
\label{subec:gIJ}
For $x=(x_1, \cdots, x_n) \ne 0$, 
 we recall $\psi_n(x)=I_n-2|x|^{-2} x\, {}^{t\!}x$
 where $|x|^2=x_1^2+\cdots+x_n^2$.  
To find the matrix-valued polynomial $H(x)$
 in \eqref{eqn:QHA}, 
 we introduce the following polynomials $g_{I J}(x)$
 for $I \in {\mathfrak{I}}_{n,i}$,
 $J \in {\mathfrak{I}}_{n-1,j}$
 with $j \in \{i-1, i\}$:
\begin{equation}
\label{eqn:gIJpr}
g_{I J}(x):=|x|^2 \langle \pr i j \circ \sigma^{(i)}(\psi(x)) e_I,
 e_J^{\vee}\rangle.  
\end{equation}
\begin{lemma}
\label{lem:171274}
$g_{I J}(x)$ is a polynomial of 
 $x= (x_1, \cdots, x_{n-1}, x_n)$
 given by
\begin{equation}
\label{eqn:gIJ}
  g_{I J}
  =
  \begin{cases}
  -S_{J I}
  \qquad
  &\text{for $j=i$}, 
\\
  (-1)^i S_{J \cup \{n\}, I}
  \qquad
  &\text{for $j=i-1$}.   
  \end{cases}
\end{equation}
\end{lemma}

\begin{proof}
The right-hand side of \eqref{eqn:gIJpr} amounts to 
\[
   |x|^2 \sum_{K \in {\mathfrak{I}}_{n,i}}(\det \psi_n(x))_{K I}
   \langle \pr i j (e_K), e_J^{\vee}\rangle.  
\]
The projection $\pr i j \colon \Exterior^i({\mathbb{C}}^n) \to \Exterior^j({\mathbb{C}}^{n-1})$
 in Section \ref{subsec:intmero}
 sends the basis $\{e_I\}$
 as follows.  
\begin{equation*}
   \pr i i(e_I)
   = 
  \begin{cases}
  e_I \qquad \text{for $n \not \in I$}, 
\\
  0\qquad \hphantom{i}\text{for $n \in I$}, 
  \end{cases}
\quad
\pr i {i-1}(e_I)
   = 
  \begin{cases}
  0 \qquad \hphantom{MMMMNi}\text{for $n \not \in I$}, 
\\
  (-1)^{i-1} e_{I \setminus \{n\}}\qquad \text{for $n \in I$}.  
  \end{cases}
\end{equation*}
Therefore,
 we get the desired result by \eqref{eqn:psidet}.  
\end{proof}

\subsection{Matrix components of $\Atbb \lambda \nu \pm {i,j}$}
\label{subec:IJA}
For $I \in {\mathfrak {I}}_{n,i}$ and $J \in {\mathfrak {I}}_{n-1,j}$, 
 we define 
$
   (\Atcal \lambda \nu \pm {i,j})_{I J} \in {\mathcal{D}}'({\mathbb{R}}^n)
$
 as the $(I,J)$-component 
 of the distribution kernel $\Atcal \lambda \nu \pm {i,j}$
 of the matrix-valued symmetry breaking operator
 $\Atbb \lambda \nu \pm {i,j}$ by 
\[
   (\Atcal \lambda \nu {\pm}{i,j})_{I J}
   :=
   \langle 
   \Atcal \lambda \nu {\pm}{i,j} (e_I), 
   e_J^{\vee}
   \rangle.  
\]
We compare the matrix-valued symmetry breaking operator
 with the {\it{shifted}} scalar-valued one.  
We set 
\begin{equation}
\label{eqn:Ascalar}
\Atcal \lambda \nu + {}
:=
a_+(\lambda,\nu)|x|^{-2\nu}|x_n|^{\lambda+\nu-n}.  
\end{equation}
This is a distribution 
 with holomorphic parameter $(\lambda,\nu) \in {\mathbb{C}}^2$, 
 and is identified with $\Atcal \lambda \nu +{0,0}$.  
\begin{lemma}
\label{lem:AIJ}
The matrix component $(\Atcal \lambda \nu {\pm}{i,j})_{I J}$  
 takes the following form:
\begin{align}
  (\Atcal \lambda \nu {+}{i,j})_{I J}
  =&
  \frac{2}{\lambda-\nu-2} g_{I J} \Atcal {\lambda-1} {\nu+1} +{}, 
\label{eqn:AijIJ}
\\
  (\Atcal \lambda \nu {-}{i,j})_{I J}
  =&
    \frac{2}{(\lambda+\nu-n)(\lambda-\nu-1)(\lambda-\nu-3)} x_n g_{I J} 
    \Atcal {\lambda-2} {\nu+1} +{}.  
\label{eqn:152292}
\end{align}
\end{lemma}
\begin{proof}
If ${\operatorname{Re}} \lambda \gg |{\operatorname{Re}} \nu|$, 
 then the definition \eqref{eqn:Aijn}
 of $\Atcal \lambda \nu {+}{i,j}$ shows that
\begin{align*}
(\Atcal \lambda \nu {+}{i,j})_{I J}
=& a_+(\lambda, \nu) |x|^{-2 \nu} |x_n|^{\lambda + \nu-n}
   \langle \pr ij \circ \sigma^{(i)}(\psi(x))(e_I), e_J^{\vee} \rangle 
\\
=& a_+(\lambda, \nu) |x|^{-2 \nu-2} |x_n|^{\lambda + \nu-n}
   g_{I J}(x), 
\end{align*}
where the second equality follows from the definition \eqref{eqn:gIJpr}
of $g_{I J}(x)$.  
In view of the definition \eqref{eqn:Ascalar}
 of the scalar-valued distribution 
 $\Atcal \lambda \nu {+}{}$, 
 the identity \eqref{eqn:AijIJ}
 holds for 
${\operatorname{Re}} \lambda \gg |{\operatorname{Re}} \nu|$.  
A similar computation tells 
 that the identity \eqref{eqn:152292} holds
 for 
${\operatorname{Re}} \lambda \gg |{\operatorname{Re}} \nu|$.  
Since $g_{I J}$ is a polynomial, 
 the right-hand sides of \eqref{eqn:AijIJ} and \eqref{eqn:152292}
 are well-defined distributions
 on ${\mathbb{R}}^n$
 that depend meromorphically on $(\lambda,\nu) \in {\mathbb{C}}^2$.  
On the other hand,
 the left-hand sides of \eqref{eqn:AijIJ} and \eqref{eqn:152292}
 depend holomorphically 
 on $(\lambda,\nu) \in {\mathbb{C}}^2$
 by Fact \ref{fact:holo}.  
Therefore, 
 the identities \eqref{eqn:AijIJ} and \eqref{eqn:152292} are proved
 for all $(\lambda,\nu) \in {\mathbb{C}}^2$
 by analytic continuation.  
\end{proof}

\subsection{Review on the scalar-valued case}
\label{subec:qAC}
In the case $i=j=0$, 
 the operator $\Atbb \lambda \nu + {0,0}$ is identified
 with scalar-valued one
\[
 \Atbb \lambda \nu + {} \colon C_c^{\infty}({\mathbb{R}}^{n})
  \to C^{\infty}({\mathbb{R}}^{n-1})
\]
with the distribution kernel 
$
\Atcal \lambda \nu + {} \equiv \Atcal \lambda \nu + {0,0}, 
$
 which was thoroughly studied
 in \cite{sbon}.  
In particular, 
we recall from \cite{xk2014}
 (see also \cite[Thm.~12.2 (2)]{sbon}) 
 the residue formula
 for the {\it{scalar-valued}} symmetry breaking operators
 as follows.  
For $(\lambda, \nu)\in {\mathbb{C}}^2$
 with $\nu-\lambda \in 2{\mathbb{N}}$, 
 we set
\begin{equation}
\label{eqn:qAC}
q_C^A(\lambda, \nu):= \frac{(-1)^{\frac {\nu-\lambda}{2}} 
                            (\frac {\nu-\lambda}{2})! 
                            \pi^{\frac{n-1}{2}}}
                           {2^{\nu-\lambda} \ \Gamma(\nu)}.  
\end{equation}
\begin{fact}
[residue formula in the scalar-valued case]
\label{fact:resAC}
Suppose $(\lambda, \nu)\in {\mathbb{C}}^2$
 satisfies $\nu-\lambda  \in 2{\mathbb{N}}$.  
Then we have
\[
   \Atbb \lambda \nu + {} = q_C^A (\lambda, \nu) \Ctbb \lambda \nu {}.  
\]
\end{fact}

\subsection{Reduction to the scalar-valued case}
We are ready to formulate an intermediate step
 for the proof of Theorem \ref{thm:153316a}
 on the residue formula
 of the matrix-valued regular symmetry breaking operators
 $\Atbb \lambda \nu \pm {i,j}$.  
We recall that $g_{I J}(x)$ is a quadratic polynomial of $x$
 (see Lemma \ref{lem:171274})
 and $\Ctcal \lambda \nu {}{}$ is a distribution on ${\mathbb{R}}^n$
 supported at the origin 
 (see \eqref{eqn:PsiJuhl}).  
Then we have the following.  

\begin{proposition}
\label{prop:AresgC}
\begin{enumerate}
\item[{\rm{(1)}}]
If $\nu-\lambda=2m$ with $m \in {\mathbb{N}}$, 
 then 
\[
  (\Atcal \lambda \nu + {i,j})_{I J}
  =
  \frac{(-1)^m m! \pi^{\frac{n-1}{2}}}{2^{2m+2} \Gamma(\nu+1)} g_{I J}(x) \Ctcal {\lambda-1}{\nu+1}{}.  
\]
\item[{\rm{(2)}}]
If $\nu-\lambda=2m+1$ with $m \in {\mathbb{N}}$, 
 then 
\[
  (\Atcal \lambda \nu - {i,j})_{I J}
  =
  \frac{(-1)^m m! \pi^{\frac{n-1}{2}}}{2^{2m+5} (\lambda+\nu-n) \Gamma(\nu+1)} 
  x_n g_{I J}(x) \Ctcal {\lambda-2}{\nu+1}{}.  
\]
\end{enumerate}
\end{proposition}
\begin{proof}
[Proof of Proposition \ref{prop:AresgC}]
(1)\enspace
By Fact \ref{fact:resAC}, 
 we have 
\[
\Atcal {\lambda-1}{\nu+1}{+}{} 
= q_C^A (\lambda-1, \nu+1)
\Ctcal {\lambda-1}{\nu+1}{}
\]
if $\nu- \lambda \in \{-2,0,2,4,\cdots\}$.  
Combining this with \eqref{eqn:AijIJ}, 
 we obtain
\[
(\Atcal \lambda \nu {+} {i,j})_{I J}
=
\frac{2}{\lambda-\nu-2}
q_C^A (\lambda-1,\nu+1)
\Ctcal{\lambda-1}{\nu+1}{}
g_{I J}.  
\]
Now a simple computation shows
 the first statement.  
\par\noindent
(2)\enspace
Suppose $\nu-\lambda =2m+1$
 with $m \in {\mathbb{N}}$.  
By Fact \ref{fact:resAC}, 
 we have 
\[
\Atcal {\lambda-2} {\nu+1} {+} {}{} 
=
q_C^A (\lambda-2,\nu+1)
\Ctcal {\lambda-2}{\nu+1}{}.  
\]
Then \eqref{eqn:152292} and Fact \ref{fact:resAC} tell that 
\begin{align*}
(\Atcal {\lambda}{\nu}{-}{i,j})_{I J}
=& 
\frac{x_n g_{I J} \Atcal {\lambda-2}{\nu+1}{+}{}}
     {2(\lambda+\nu-n)(m+1)(m+2)}
\\
=&
\frac{q_C^A (\lambda-2,\nu+1)}
     {2(\lambda+\nu-n)(m+1)(m+2)}
x_n g_{I J} \Ctcal {\lambda-2}{\nu+1}{}.  
\end{align*}
Thus the second statement is shown.  
\end{proof}

In order to find a closed expression
 of the right-hand sides of the formul{\ae}
 in Proposition \ref{prop:AresgC}, 
 we shall apply in the next section
 the identities in ${\mathbb{C}}[x,\frac{\partial}{\partial x}]/{\mathcal{J}}$
 proved in Section \ref{subsec:Weylalg}.  
In particular,
 we shall see
 in Proposition \ref{prop:153294}
 that a 
$
   {\operatorname{Hom}}_{\mathbb{C}}(\Exterior^i({\mathbb{C}}^n), 
                                     \Exterior^j({\mathbb{C}}^{n-1}))
$-valued distribution
 on ${\mathbb{R}}^n$
 ($j=i-1,i$) 
whose $(I,J)$-component
 is equal to $g_{I J}(x) \Ctcal {\lambda-1}{\nu+1}{}$
 recovers
 the differential symmetry breaking operator
\[
  \Cbb {\lambda}{\nu}{i,j}
  \colon 
  {\mathcal{E}}^i({\mathbb{R}}^n) \to {\mathcal{E}}^j({\mathbb{R}}^{n-1})
\]
 given in Section \ref{subsec:residue}.

\section{Proof of Theorem \ref{thm:153316a}}
\label{subsec:respf}
In this section we complete the proof 
 of the residue formula
 of the {\it{matrix-valued}} symmetry breaking operator
 $\Atbb \lambda \nu \pm {i,j}$
 by comparing the $(I,J)$-component
 of the equation \eqref{eqn:resAvecC}.

\subsection{Matrix components
of matrix-valued differential symmetry breaking operators
 $\Cbb \lambda \nu {i,j}$}
\label{subsec:IJC}
We recall the matrix-valued differential operator from
 Section \ref{subsec:residue}:
\[
  \Cbb \lambda \nu {i,j} \colon {\mathcal{E}}^i({\mathbb{R}}^n)
  \to {\mathcal{E}}^j({\mathbb{R}}^{n-1})
  \quad
  (j=i-1,i).  
\]

For $I \in {\mathfrak {I}}_{n,i}$ and $J \in {\mathfrak {I}}_{n-1,j}$, 
 we define a linear map 
$(\Cbb \lambda \nu {i,j})_{I J} \colon C^{\infty}({\mathbb{R}}^n) \to C^{\infty}({\mathbb{R}}^{n-1})$ as the $(I,J)$-component
 of the differential operator $\Cbb \lambda \nu {i,j}$, 
 which is characterized by the formula
\[
   \Cbb \lambda \nu {i,j}(f(x) d x_I)
   =
   \sum_{J \in {\mathfrak{I}}_{n-1,j}} ((\Cbb \lambda \nu {i,j})_{I J}f) d x_J.
\]  
Then there exist uniquely distributions
 $(\Ccal \lambda \nu {i,j})_{I J}$ on ${\mathbb{R}}^n$
 supported at the origin
 such that 
\[
   (\Cbb \lambda \nu {i,j})_{I J}
   =
   \operatorname{Rest}_{x_n=0} 
   \circ
   (\Ccal \lambda \nu {i,j})_{I J} \ast.  
\]
The explicit formul{\ae} of $(\Cbb \lambda \nu {i,j})_{I J}$
 (or equivalently, of $(\Ccal \lambda \nu {i,j})_{I J}$)
 are given in Facts \ref{fact:153284} and \ref{fact:161478} below.

Both the distributions $(\Ccal \lambda \nu {i,j})_{I J}$ and 
 $g_{I J}(x)\Ctcal {\lambda-1}{\nu+1}{}$ are distributions 
on ${\mathbb{R}}^n$ supported at the origin.  
We give its relationship as follows.   
\begin{proposition}
\label{prop:153294}
Let $j=i$ or $i-1$.  
For any $I \in {\mathcal{I}}_{n,i}$
 and $J \in {\mathcal{I}}_{n-1,j}$, 
 we have 
\begin{equation}
\label{eqn:gIJC}
  g_{I J}(x) \Ctcal{\lambda-1}{\nu+1}{}
  =
  8 (-1)^{i-j} ({\mathcal{C}}_{\lambda,\nu}^{i,j})_{I J}.  
\end{equation}
\end{proposition}

The next two subsections will be devoted to the proof of Proposition \ref{prop:153294}
 by using the identities
 in Section \ref{subsec:Weylalg}.  
The cases $j=i$ and $j=i-1$ are treated separately 
 in Sections \ref{subsec:resAC} and \ref{subsec:resAC2}.  
In Proposition \ref{prop:AresgC}, 
 we have related the left-hand side of \eqref{eqn:gIJC}
 with the $(I,J)$-component
 of the regular symmetry breaking operator
 $\Atbb \lambda \nu {\pm} {i,j}$.  
Thus we shall complete the proof of Theorem \ref{thm:153316a}
 based on Proposition \ref{prop:153294}.  
The last step will be given in Section \ref{subsec:pfmain}.

\subsection{Proof of Proposition \ref{prop:153294} for $j=i$}
\label{subsec:resAC}
Let $j=i$.  
Suppose $I \in {\mathcal{I}}_{n,i}$ 
and $J \in {\mathcal{I}}_{n-1,i}$.  
The main cases will be the following.  
\par\noindent
{\bf{Case 1.}}\enspace 
$n \not \in I$, $J =I$.  
\par\noindent
{\bf{Case 2.}}\enspace 
$n \not \in I$, $\#(J \setminus I) =1$. 
We may write $I= K \cup \{p\}$, $J = K \cup \{q\}$.  
 \par\noindent
{\bf{Case 3.}}\enspace 
$n \in I$, $\#(J \setminus I) =1$. 
We may write $I= K \cup \{n\}$, $J = K \cup \{q\}$.   

\begin{fact}
[{\cite[Lem.~9.6]{KKP}}]
\label{fact:153284}
The $(I,J)$-component $(\Cbb \lambda \nu {i,i})_{I J}$
 of the differential operator 
$
   \Cbb \lambda \nu {i,i} \colon {\mathcal{E}}^i({\mathbb{R}}^{n})
   \to 
   {\mathcal{E}}^i({\mathbb{R}}^{n-1})
$ is equal to 
\begin{enumerate}
\item[]
{\rm{{\bf{Case 1.}}}}\enspace
$- 
\Ctbb{\lambda+1}{\nu-1}{}
(\sum_{p \in I}\frac{\partial^2}{\partial x_p^2})
+
\frac 1 2 (\nu - i)
\Ctbb{\lambda}{\nu}{}, 
$
\item[]
{\rm{{\bf{Case 2.}}}}\enspace
$ - \operatorname{sgn}(I;p,q) 
\Ctbb{\lambda+1}{\nu-1}{}
\frac{\partial^2}{\partial x_p \partial x_q}, 
$
\item[]
{\rm{{\bf{Case 3.}}}}\enspace
$ - \operatorname{sgn}(I;q,n)
\gamma(\lambda-\frac{n-1}{2}, \nu-\lambda)
\Ctbb{\lambda+1}{\nu}{}
\frac{\partial}{\partial x_q}, 
$
\end{enumerate}
and is zero otherwise.  
\end{fact}

\begin{proof}
[Proof of Proposition \ref{prop:153294} for $j=i$]
It is readily seen
 that the both sides of \eqref{eqn:gIJC} vanish
 unless $(I,J)$ belongs to Cases 1--3.  
{}From now,
 we focus on Cases 1--3.  
\par\noindent
{{\bf{Case 1.}}}\enspace 
By \eqref{eqn:SIJ}, 
 the polynomial $g_{IJ}(x)$ in \eqref{eqn:gIJ}
 amounts to 
$
   g_{IJ}(x)=x_n^2 - \sum_{k=1}^{n-1}\varepsilon_I(k) x_k^2.  
$
Then a small computation using \eqref{eqn:1523114} and \eqref{eqn:152403} shows
 the following equalities 
 in ${\mathcal{D}}_{ \{0\} }'({\mathbb{R}}^n)$:
\begin{equation*}
g_{IJ}(x) \Ctcal{\lambda-1}{\nu+1}{}
\\
=
4(\nu-i) \Ctcal{\lambda}{\nu}{}
-8(\sum_{k \in I} \frac{\partial^2}{\partial x_k^2}) 
\Ctcal{\lambda+1}{\nu-1}{}.  
\end{equation*}
By Fact \ref{fact:153284} in Case 1
 and \eqref{eqn:160205}, 
 we get 
\begin{equation}
\label{eqn:R8}
 {\operatorname{Rest}}_{x_n=0} \circ (g_{IJ} \Ctcal {\lambda-1}{\nu+1}{}) \ast
=8 (\Cbb {\lambda}{\nu}{i,i})_{IJ}.  
\end{equation}
\vskip 0.5pc
\par\noindent
{{\bf{Case 2.}}}\enspace 
In this case, 
$
   g_{IJ}(x)= -2 \operatorname{sgn} (K; p,q)x_p x_q.  
$
By \eqref{eqn:1523101}, 
\begin{align*}
g_{IJ} (x) \Ctcal{\lambda-1}{\nu+1}{}
=&
-2 \operatorname{sgn} (K; p,q)x_p x_q \Ctcal{\lambda-1}{\nu+1}{}
\\
=&-8\operatorname{sgn} (K; p,q) \frac{\partial^2}{\partial x_p \partial x_q} \Ctcal{\lambda+1}{\nu-1}{}.  
\end{align*}
Comparing this with Fact \ref{fact:153284} in Case 2, 
 we get \eqref{eqn:R8} in this case.  
\vskip 0.5pc
\par\noindent
{{\bf{Case 3.}}}\enspace 
Suppose $n \in I$ and $I=K \cup \{n\}$, 
 $J=K \cup \{q\}$.  
Then 
\[
   g_{IJ}(x)= 2 \operatorname{sgn} (K; q,n)x_q x_n.  
\]
By \eqref{eqn:152383n}, 
we have
\begin{align*}
g_{IJ}(x) \Ctcal{\lambda-1}{\nu+1}{}
=&
2 \operatorname{sgn} (K; q,n)x_q x_n \Ctcal{\lambda-1}{\nu+1}{}
\\
=& -8 \gamma(\lambda - \frac {n-1}{2}, \nu - \lambda)
   \operatorname{sgn} (K; q,n) \frac{\partial}{\partial x_q}
   \Ctcal{\lambda+1}{\nu}{}.  
\end{align*}
Again by Fact \ref{fact:153284}, 
 we get \eqref{eqn:R8} in this case.  
\end{proof}

\subsection{Proof of Proposition \ref{prop:153294}
 for $j=i-1$}
\label{subsec:resAC2}
In this section,
 we give a proof 
 of Proposition \ref{prop:153294}
 for $j=i-1$.  
Suppose $I \in {\mathcal{I}}_{n,i}$ 
and $J \in {\mathcal{I}}_{n-1,i-1}$.  
The main cases will be the following.  
\par\noindent
{{\bf{Case 1.}}}\enspace 
$n \in I$, $J =I \setminus \{ n \}$.  
\par\noindent
{{\bf{Case 2.}}}\enspace 
$n \in I$, $\#(J \setminus I) =1$. 
We may write $I= K \cup \{p, n\}$, $J = K \cup \{q\}$.  
\par\noindent
{{\bf{Case 3.}}}\enspace 
$n \not \in I$, $J \subset I$. 
We may write $I= J \cup \{p\}$.   

\begin{fact}
[{\cite[Lem. 9.5]{KKP}}]
\label{fact:161478}
The $(I,J)$-component $(\Cbb \lambda \nu {i,i-1})_{I J}$
 of the differential operator 
 $\Cbb \lambda \nu {i,i-1}\colon {\mathcal{E}}^i({\mathbb{R}}^{n})\to {\mathcal{E}}^{i-1}({\mathbb{R}}^{n-1})$ is equal to 
\begin{enumerate}
\item[]
{\rm{{\bf{Case 1.}}}}\enspace
$(-1)^{i-1}
(-\Ctbb{\lambda+1}{\nu-1}{}
(\sum_{p \not \in I}\frac{\partial^2}{\partial x_p^2})
+
\frac {\nu + i - n} 2
\Ctbb{\lambda}{\nu}{}), 
$
\item[]
{\rm{{\bf{Case 2.}}}}\enspace
$ (-1)^{i-1} \operatorname{sgn}(I;p,q) 
\Ctbb{\lambda+1}{\nu-1}{}
\frac{\partial^2}{\partial x_p \partial x_q}, 
$
\item[]
{\rm{{\bf{Case 3.}}}}\enspace
$\operatorname{sgn}(I;p)
\gamma(\lambda-\frac{n-1}{2}, \nu-\lambda)
\Ctbb{\lambda+1}{\nu}{}
\frac{\partial}{\partial x_p}, 
$
\end{enumerate}
and is zero otherwise.  
\end{fact}

\begin{proof}
[Proof of Proposition \ref{prop:153294} for $j=i-1$]
{}~~~{}
\par
{\bf{Case 1.}} \enspace
By \eqref{eqn:SIJ}, 
we have 
$
g_{IJ}(x)=(-1)^i(x_n^2 + \sum_{k=1}^{n-1} \varepsilon_I(k) x_k^2).  
$
Then \eqref{eqn:1523114} and \eqref{eqn:152403} tell
 that 
\begin{equation*}
(-1)^i g_{IJ}(x) \Ctcal {\lambda-1}{\nu+1}{}
= 4 (\nu -n +i)\Ctcal {\lambda}{\nu}{}
 -8 \sum_{k \not \in I} \frac{\partial^2}{\partial x_k^2} 
 \Ctcal {\lambda+1}{\nu-1}{}.  
\end{equation*}
Comparing this with Fact \ref{fact:161478} in Case 1, 
 we get 
\begin{equation}
\label{eqn:R8b}
\operatorname{Rest}_{x_n=0} \circ (g_{IJ} \Ctcal {\lambda-1}{\nu+1}{})\ast
= -8 (\Cbb {\lambda}{\nu}{i,i-1})_{IJ}.  
\end{equation}

\vskip 0.5pc
{\bf{Case 2.}} \enspace
In this case, 
$
g_{IJ}(x)= 2 (-1)^i \operatorname{sgn}(K;p,q) x_p x_q.  
$
By \eqref{eqn:1523101}, 
we have 
\[
g_{IJ}(x) \Ctcal {\lambda-1}{\nu+1}{}
= 8 (-1)^{i} \operatorname{sgn}(K;p,q) 
  \frac{\partial^2}{\partial x_p \partial x_q}
   \Ctcal {\lambda+1}{\nu-1}{}.  
\]
Hence, 
 Fact \ref{fact:161478} in Case 2 implies \eqref{eqn:R8b}.

\vskip 0.5pc
{\bf{Case 3.}} \enspace
In this case, 
$
g_{IJ}(x)= -2 \operatorname{sgn}(I;p) x_p x_n.  
$
Then \eqref{eqn:152383n} implies 
\[
g_{IJ}(x) \Ctcal {\lambda-1}{\nu+1}{}
= -8 \gamma(\lambda-\frac{n-1}{2}, \nu-\lambda)
    \operatorname{sgn}(I;p) 
    \frac{\partial}{\partial x_p}
    \Ctcal {\lambda+1}{\nu}{}.  
\]
Hence Fact \ref{fact:161478} in Case 3 shows \eqref{eqn:R8b}.  
\end{proof}

\subsection{Proof of Theorem \ref{thm:153316a}}
\label{subsec:pfmain}
We are ready to complete the proof of Theorem \ref{thm:153316a}.

\begin{proof}
[Proof of Theorem \ref{thm:153316a}]
Let $\gamma=0$.  
Theorem \ref{thm:153316a} in this case 
 follows from Propositions \ref{prop:AresgC} and \ref{prop:153294}.  

Next, 
let $\gamma=1$.  
By \eqref{eqn:1532101} and Proposition \ref{prop:153294}, 
\begin{align*}
x_n g_{I J} \Ctcal {\lambda-2}{\nu+1}{}
=& 
(n-\lambda-\nu) g_{I J} \Ctcal {\lambda-1}{\nu+1}{}
\\
=&
8 (-1)^{i-j+1} (\lambda+\nu-n)
(\Ccal {\lambda}{\nu}{i,j})_{I J}.  
\end{align*}
Hence Theorem \ref{thm:153316a} for $\gamma=1$ is shown.  
\end{proof}

\section{Vanishing condition of the symmetry breaking operator
 $\Atbb \lambda \nu {\pm} {i,j}$}
\label{subsec:Avanish}

As a corollary of Theorem \ref{thm:153316a}, 
 we can determine the (isolated) zeros of the analytic continuation
 $\Atbb \lambda \nu {\pm} {i,j}$
 of the integral operators.  
Following the notation as in \cite[Chap.~1]{sbon}, 
 we define two subsets in ${\mathbb{Z}}^2$ as below:
\begin{align*}
L_{{\operatorname{even}}}&:=\{(-i,-j): 0 \le j \le i \text{ and } i \equiv j \mod 2\},
\\
L_{{\operatorname{odd}}}&:=\{(-i,-j): 0 \le j \le i \text{ and } i \equiv j+1 \mod 2\}.   
\end{align*}
\begin{theorem}
\label{thm:Avanish}
\begin{enumerate}
\item[{\rm{(1)}}]
Suppose $\nu - \lambda \in 2 {\mathbb{N}}$.  
\newline
$\Atbb \lambda \nu {+} {i,i}=0$
if and only if
\begin{equation*}
(\lambda, \nu)\in 
\begin{cases}
L_{\operatorname{even}}
&\text{for $i=0$}, 
\\
(L_{\operatorname{even}} \setminus \{\nu =0\})\cup \{(i,i)\}
\quad
&\text{for $1 \le i \le n-1$}.   
\end{cases}
\end{equation*}
$\Atbb \lambda \nu {+} {i,i-1}=0$
if and only if
\begin{equation*}
(\lambda, \nu)\in 
\begin{cases}
(L_{\operatorname{even}} \setminus \{\nu =0\})\cup \{(n-i,n-i)\}
\quad
&\text{for $1 \le i \le n-1$}, 
\\
L_{\operatorname{even}}
&\text{for $i=n$}.   
\end{cases}
\end{equation*}
\item[{\rm{(2)}}]
Suppose $\nu - \lambda \in 2 {\mathbb{N}}+1$.  
\newline
$\Atbb \lambda \nu {-} {i,i}=0$
if and only if
\begin{equation*}
(\lambda, \nu)\in 
\begin{cases}
L_{\operatorname{odd}}
&\text{for $i=0$}, 
\\
L_{\operatorname{odd}} \setminus \{\nu =0\}
\quad
&\text{for $1 \le i \le n-1$}.   
\end{cases}
\end{equation*}
$\Atbb \lambda \nu {-} {i,i-1}=0$
if and only if
\begin{equation*}
(\lambda, \nu)\in 
\begin{cases}
L_{\operatorname{odd}} \setminus \{\nu =0\}
\quad
&\text{for $1 \le i \le n-1$}, 
\\
L_{\operatorname{odd}}
&\text{for $i=n$}.   
\end{cases}
\end{equation*}
\end{enumerate}
\end{theorem}

Owing to the residue formula
 (Theorem \ref{thm:153316a}), 
we can reduce the proof of Theorem \ref{thm:Avanish}
 to an easy question,
 that is, 
 to find a necessary and sufficient condition
 for the matrix-valued differential symmetry breaking operators
 $\Cbb \lambda \nu {i,j}$ to vanish.  
The latter condition is verified immediately by the formula
 for the $(I,J)$-component 
 of $\Cbb \lambda \nu {i,j}$
 (see Facts \ref{fact:153284} and \ref{fact:161478}), 
 and is described as follows:
\begin{lemma}
[{\cite[Prop.~1.4 and p.~23]{KKP}}]
\label{lem:Cij0}
Suppose $(\lambda, \nu)\in {\mathbb{C}}^2$
 with $\nu-\lambda \in {\mathbb{N}}$.  
\begin{enumerate}
\item[{\rm{(1)}}]
Let $1 \le i \le n$.  
Then $\Cbb \lambda \nu {i,i}$ vanishes
 if and only if $\lambda = \nu=i$
 or $\nu=i=0$.  
\item[{\rm{(2)}}]
Let $1 \le i \le n-1$.  
Then $\Cbb \lambda \nu {i,i-1}$ vanishes
 if and only if $\lambda = \nu=n-i$
 or $\nu=n-i=0$.  
\end{enumerate}
\end{lemma}

\begin{proof}
[Proof of Theorem \ref{thm:Avanish}]
Suppose
\begin{equation*}
\nu-\lambda \in 2 {\mathbb{N}}
\quad
(\varepsilon=+)
\quad
\text{or}
\quad
\nu-\lambda \in 2 {\mathbb{N}}+1
\quad
(\varepsilon=-).  
\end{equation*}
In either case, 
 the residue formula in Theorem \ref{thm:153316a} asserts that 
\[
\Atbb \lambda \nu \varepsilon {i,j} = 
\frac{c}{\Gamma(\nu+1)}\Cbb \lambda \nu {i,j}, 
\]
for some $c \ne 0$.  
Therefore we have 
\[
  \Atbb \lambda \nu \varepsilon {i,j} = 0
\quad
\text{if and only if}
\quad
\nu \in -{\mathbb{N}}_+
\text{ or }
\Cbb \lambda \nu {i,j}=0.  
\]
In light of Lemma \ref{lem:Cij0}, 
 we conclude Theorem \ref{thm:Avanish}.  
\end{proof}
\begin{remark}
\label{rem:Avanish}
By the general results (see \cite{sbonvec}), 
 $\Atbb \lambda \nu \varepsilon {i,j}$ vanishes
 only if 
\[
  \text{
 $\nu -\lambda \in 2 {\mathbb{N}}\quad(\varepsilon =+)$
 \quad
  or 
 \quad
  $\nu -\lambda \in 2 {\mathbb{N}} + 1\quad(\varepsilon =-)$.  
}
\]
Hence Theorem \ref{thm:Avanish} determines precisely
 when the regular symmetry breaking operator $\Atbb \lambda \nu \pm {i,j}$ vanishes.  
\end{remark}


\end{document}